\documentclass{amsart}
\usepackage{amsfonts,amssymb,amsthm,cite,amsmath,amstext,amsbsy,bm}
\usepackage[toc,page]{appendix}
\usepackage{paralist}
\usepackage{datetime}
\usepackage{color}
\usepackage[dvipsnames]{xcolor}
\usepackage[colorlinks,linkcolor=Bittersweet,citecolor=Cerulean]{hyperref}
\usepackage{mathrsfs}
\usepackage{mathtools}
\usepackage{euscript}
\usepackage[all]{xy}
\usepackage{dsfont}
\usepackage{graphicx}

\usepackage[text={170mm,260mm},centering]{geometry}
\usepackage{tikz-cd}

\usepackage{thmtools}
\usepackage{enumitem}
\usepackage{letltxmacro}

\declaretheorem[
name=Theorem,
Refname={Theorem,Theorems},
numberwithin=section]{thm}
\declaretheorem[
name=Proposition,
Refname={Proposition,Propositions},
sibling=thm]{proposition}
\declaretheorem[
name=Definition,
Refname={Definition,Definitions},
sibling=thm]{dfn}
\declaretheorem[
name=Corollary,
Refname={Corollary,Corollaries},
sibling=thm]{cor}

\usepackage{nameref}
\usepackage[capitalize]{cleveref}

\theoremstyle{plain}
\newlist{thmlist}{enumerate}{1}
\setlist[thmlist]{label=(\roman{thmlisti}), ref=\thethm.(\roman{thmlisti}),noitemsep}

\addtotheorempostheadhook[thm]{\crefalias{thmlisti}{thm}}

\addtotheorempostheadhook[proposition]{\crefalias{thmlisti}{proposition}}

\addtotheorempostheadhook[dfn]{\crefalias{thmlisti}{dfn}}

\def\A{\mathbb{A}}
\def\ep{\varepsilon}
\def\O{\mathscr{O}}
\def\P{\mathbb{P}}
\def\V{\mathbb{V}}
\def\Si{\Sigma}
\def\D{D}
\def\c{n}
\def\qb{\mathbb{Q}}
\def\p{p}
\def\pr{{\rm pr}}
\def\X{\mathscr{X}}
\def\sm{{\rm sm}}
\def\Supp{{\rm Supp}}
\def\res{{\rm res}}
\def\Res{{\rm Res}}
\def\G{\mathbb{G}}
\def\codim{{\rm codim}}
\def\lb{\mathbb{L}}
\def\rank{{\rm rank}}
\usepackage{cleveref}
\crefname{lem}{Lemma}{Lemmas}
\crefname{thm}{Theorem}{Theorems}
\crefname{proposition}{Proposition}{Propositions}
\crefname{dfn}{Definition}{Definitions}
\crefname{rem}{Remark}{Remarks}
\crefname{cor}{Corollary}{Corollaries}
\crefname{corx}{Corollary}{Corollaries}

\makeatletter
\newcommand*{\rom}[1]{\expandafter\@slowromancap\romannumeral #1@}
\makeatother

\makeatletter
\newsavebox{\@brx}
\newcommand{\llangle}[1][]{\savebox{\@brx}{\(\m@th{#1\langle}\)}%
  \mathopen{\copy\@brx\kern-0.5\wd\@brx\usebox{\@brx}}}
\newcommand{\rrangle}[1][]{\savebox{\@brx}{\(\m@th{#1\rangle}\)}%
  \mathclose{\copy\@brx\kern-0.5\wd\@brx\usebox{\@brx}}}
\makeatother

\newtheorem{thmx}{Theorem}

\numberwithin{equation}{section}

\definecolor{plum}{rgb}{0.8,0.2,0.8}

\title[Positivity of the log  cotangent bundle]{\textbf{\sc On the positivity of the logarithmic  cotangent bundle}}

\author{Damian Brotbek}

\address{Institut de Recherche Math\'ematique Avanc\'ee\\
	Universit\'e de Strasbourg}
\email{brotbek@math.unistra.fr}
\author{Ya Deng}
\address{Institut de Recherche Math\'ematique Avanc\'ee\\
	Universit\'e de Strasbourg}
\email{deng@math.unistra.fr;\ dengya.math@gmail.com}
\begin{document}
	\date{\today}
\maketitle



\newtheorem{rem}[thm]{Remark}
\newtheorem{problem}[thm]{Problem}
\newtheorem{conjecture}[thm]{Conjecture}
\newtheorem{lem}[thm]{Lemma}
\newtheorem{example}[thm]{Example}
\newtheorem{propx}[thmx]{Proposition}

\newtheorem{claim}{\underline{\textbf{Claim}}}[section]

\def\oc{\mathscr{O}} \def\oce{\mathcal{O}_E} \def\xc{\mathcal{X}}
\def\ac{\mathcal{A}} \def\rc{\mathcal{R}} \def\mc{\mathcal{M}}
\def\wc{\mathcal{W}} \def\fc{\mathcal{F}} \def\cf{\mathcal{C_F^+}}
\def\jc{\mathcal{J}}
\def\jr{\mathscr{J}}
\def\mr{\mathscr{M}}
\def\ic{\mathcal{I}}
\def\kc{\mathcal{K}}
\def\lc{\mathcal{L}}
\def\vc{\mathcal{V}}
\def\xc{\mathcal{X}}
\def\yc{\mathcal{Y}}
\def\zc{\mathcal{Z}}
\def\af{\mathbf{a}}
\def\df{\mathbf{d}}
\def\pr{{\rm pr}}

\def\cb{\mathbb{C}}
\def\ib{\mathbb{I}}
\def\ab{\mathbb{A}}
\def\lbb{\mathbb{L}}
\def\vol{\operatorname{vol}}
\def\ord{\operatorname{ord}}
\def\Im{\operatorname{Im}}
\def\dm{\mathrm{d}}
\def\isc{\mathscr{I}}

\def\as{{a^\star}} \def\es{e^\star}

\def\tl{\widetilde} \def\tly{{Y}} \def\om{\omega}

\def\cb{\mathbb{C}} \def\nb{\mathbb{N}}
\def\gb{\mathbb{G}} \def\nbs{\mathbb{N}^\star}
\def\pb{\mathbb{P}} \def\pbe{\mathbb{P}(E)} \def\rb{\mathbb{R}}
\def\zbb{\mathbb{Z}}
\def\ys{\mathscr{Y}}
\def\js{\mathscr{J}}
\def\is{\mathscr{I}}
\def\ls{\mathscr{L}}
\def\ds{\mathscr{D}}
\def\hs{\mathscr{H}}
\def\xs{\mathscr{X}}
\def\zs{\mathscr{Z}}
\def\grs{{\rm{Gr}}_{k+1}(V)}
\def\ebf{\mathbf{e}}
\def\yf{\mathfrak{Y}}
\def\gf{\mathbf{G}}
\def\bff{\mathbf{b}}
\def\wk{\mathfrak{w}}
\def\bk{\mathfrak{b}}
\def\ik{\mathfrak{I}}
\def\sif{\bm{\sigma}}

\def\hb{\bold{H}} \def\fb{\bold{F}} \def\eb{\bold{E}}
\def\pbb{\bold{P}}

\def\nab{{\nabla}} \def\n{|\!|} \def\spec{\textrm{Spec}\,}
\def\cinf{\mathcal{C}_\infty} \def\d{\partial}
\def\db{{\partial}}
\def\hess{\sqrt{-1}\partial{\partial}}
\def\zb{{z}} \def\lra{\longrightarrow}
\def\tb{{t}}
\def\sn{\sqrt{-1}}
\begin{abstract}
The aim of this work is to construct examples of pairs whose logarithmic cotangent bundles have  strong positivity properties.
These examples are constructed from any smooth \(n\)-dimensional complex projective varieties by considering the sum of at least \(n\) general sufficiently ample hypersurfaces.
\end{abstract}



\section{Introduction}
To every smooth pair \((X,D)\) (a smooth projective variety \(X\) with a simple normal crossing divisor \(D\)), one can associate a logarithmic cotangent bundle \(\Omega_X(\log D)\). The properties of this bundle have important consequences concerning the geometry of the pair \((X,D)\). In the present paper, we focus on the positivity properties  of the logarithmic cotangent bundle and in particular on the study in \(X_1(D):=\P(\Omega_X(\log D))\) of the augmented base locus  \(\mathbf{B}_+\big(\O_{X_1(D)}(1)\big)\) of the tautological line bundle \(\O_{X_1(D)}(1)\). There are now several examples of pairs \((X,D)\) for which this base locus is not the entire space \(X_1(D)\) \cite{BPW17,V-Z02,Rou03,Bru16,Cad16,B-C17,Nog86}, but we lack  examples of pairs for which one has a stronger control of \(\mathbf{B}_+\big(\O_{X_1(D)}(1)\big)\) when \(\Omega_X\) itself does not have a strong positivity property.

It should be noted that  if \(D\neq\varnothing\) and \(\dim X\geqslant 2\),  this augmented base locus is never empty. In particular, the logarithmic cotangent bundle is not ample. This follows from the existence of residue maps which induce a trivial quotient of the logarithmic cotangent bundle when restricted to the different components of \(D\). Each such quotient produces a positive dimensional subvariety of \(\mathbf{B}_+(\O_{X_1(D)}(1))\).

It is therefore natural to wonder if these are the only general obstructions to the ampleness of the logarithmic cotangent bundle. In the present paper we show that this is the case by producing examples for which the augmented base locus of the tautological line bundle on \(X_1(D)\) is equal to the reunion of all obstructions induced by the residue maps, in which case we say that the logarithmic cotangent bundle is \emph{almost ample}.

\begin{thmx}\label{cor:simple}
Let  $Y$ be a smooth projective variety of dimension $n$ and let \(c\geqslant n\).  Let \(L\) be a very ample line bundle on \(Y\). For any \(m\geqslant (4n)^{n+2}\) and  for general  hypersurfaces \(H_1,\dots,H_c\in |L^m|\), writing \(D=\sum_{i=1}^cH_i\), the  logarithmic cotangent bundle $\Omega_Y(\log D)$ is almost ample. 
\end{thmx}
In this statement, the case \(c=n\) is the critical case.  In  \cref{apprendix a} we shall see that for \(Y=\P^n\), and for \(c<n\), the logarithmic cotangent bundle can not be almost ample as it is not even big. Therefore, in this generality, our result is optimal on \(c\). Our main result  is a special case of \cref{main} (see \cref{se:Main}), in which we  also consider   cases where the degrees of the different hypersurfaces are different. 

The only  result of this type  we are aware of is due to Noguchi \cite{Nog86}, who proved that for pair \((\P^2,D)\), where \(D\) is the sum of \(6\) lines in general position, the logarithmic cotangent bundle satisfies a similar positivity property. \\

Let us mention  that \Cref{cor:simple} has a straightforward implication regarding the hyperbolicity of \(Y\setminus D\). In fact it implies that \(Y\setminus D\) does not contain any entire curve, where we recall that an entire curve in \(Y\setminus D\) is a non constant holomorphic map \(f:\cb\to Y\setminus D\). Moreover, one can actually prove that \(Y\setminus D\) is hyperbolically embedded in \(Y\) (see \cite{Kob98} for the definition). These results are not new and follow from the work \cite{Dar15} or \cite{Siu15}, but while these works rely on higher order jet space technics, here we just rely on jets of order one. Let us  put this in perspective with previously known results and conjectures. 

A conjecture of Kobayashi states that \emph{the complement of a general hypersurface of high degree in \(\P^n\) is Kobayashi hyperbolic}.  In view of a  result of Green \cite{Gre77} and the compact counterpart of the Kobayashi conjecture, this would imply that the the complement of a general hypersurface is in fact hyperbolically embedded in \(\P^n\). This conjecture was only proved recently by Siu \cite{Siu15}. Before that, the effective algebraic degeneracy of entire curve in these complements was proved by Darondeau \cite{Dar15} building on ideas of Voisin \cite{Voi96,Voi98}, Siu \cite{Siu04}, Diverio, Merker and Rousseau \cite{DMR10}. 

On the other hand,  the logarithmic  Green-Griffiths conjecture stipulates that if \((X,D)\) is a pair of log general type (i.e. \(K_X(D)\) is big), then every entire curve in \(X\setminus D\) is algebraically degenerate (one even expects there to be an algebraic subset \(Z\subsetneq X\) containing the image of all entire curves). One of the main general results towards this conjecture was obtained by the work of Noguchi, Winkelmann and Yamanoi  \cite{NWY07,NWY08,NWY13} and Lu and Winkelmann \cite{L-W12} who  proved that if \((X,D)\) is a pair of log general type with logarithmic irregularity \(h^0(X,\Omega_X(\log D))\geqslant \dim X\), then every entire curve in \(X\setminus D\) is algebraically degenerate. This improves and generalizes a classical result of Bloch and Ochiai. 
When \(X=\P^n\) and if \(D\) has \(n\) components, then, as we recall in \cref{apprendix a}, the logarithmic irregularity of the pair \((\P^n,D)\) is \(n-1\), and therefore the situation of \cref{cor:simple}  lies just outside the cases covered by the work of these authors.\\

The first step when one tries to control the augmented base locus of the tautological line bundle \(\O_{X_1(D)}(1)\) is to prove that this line bundle is big and therefore to construct some logarithmic symmetric differential forms on the pair \((X,D)\).
 Under some circumstances, as in \cite{Rou03} one can use a Riemann-Roch argument. A more differential geometric approach, motivated by the study of hyperbolicity properties of moduli spaces, is sometimes possible. It relies on the curvature properties of some  metric on \(X\setminus D\). This approach has proved very fruitful when for instance \(X\setminus D\) carries a variation of Hodge structure or is the base of a suitable  family of smooth varieties \cite{Zuo00,V-Z02,BPW17,Bru16,B-C17,Cad16}, but we don't know if the pairs considered in \cref{cor:simple} have such a property. 

In the present article, in order to construct logarithmic symmetric differential forms on the pairs under consideration, we follow  the general strategy of \cite{BD15,Bro17}. In fact, our main result can be seen as a logarithmic analogue of a  result of Darondeau and the first named author \cite{BD15}  on a conjecture of Debarre, taking into account the effective results of the second named author \cite{Den16,Den17}. In \cite{Deb05}, Debarre conjectured that \emph{if \(X\subset \P^n\) is  a general complete intersection of sufficiently high multidegree and such that \(\codim_{\P^n}(X)\geqslant \dim X\), then the cotangent bundle \(\Omega_X\) is ample}. This conjecture was established by Xie \cite{Xie15} with an effective uniform lower bound on the degree and independently, in the work \cite{BD15} with no effective bound and moreover an arithmetic condition on the multidegree, similar to the one appearing in \cref{main}. In \cite{Den17}, among other results, the second named was able to make the argument of \cite{BD15} effective, and improved the bounds obtained in \cite{Xie15}.\\

Let us now outline the proof  of \cref{cor:simple}.
First we reduce ourselves to the case \(c=n\), then, the main idea of the proof is to deduce the result for general pairs from the construction of a particular example \((Y,D)\)  satisfying a stronger Zariski open property. Because the logarithmic cotangent bundle can not be ample, we can not use the Zariski openness of ampleness in such a direct way. Instead, one has to take into account the obstructions from the residues  by considering a ``universal'' modification \(\widehat{\P}(\Omega_Y(\log D))\) of \({\P}(\Omega_Y(\log D))\) and prove a suitable ampleness statement on this modification. 

Such examples are constructed by considering hypersurfaces \(H_1,\dots,H_n\) being defined as suitable deformations of Fermat type hypersurfaces. For simplicity, until the end of this introduction, let us restrict ourselves to the case \(Y=\P^n\) and \(L=\O_{\P^n}(1)\).  For \(i\in \{1,\dots,n\}\), we consider  the hypersurface \(H_i\subset \P^n\) defined by a polynomial of the form 
 \[F_i=\sum_{|I|=\delta}a_{i,I}y^{(r+1)I},\]
where \(\delta\in \nb^*\), \(r\in \nb\) is large enough,  the \(a_{i,I}\)'s are   polynomials of degree \(\ep\geqslant 1\) in \(\cb[y_0,\dots,y_n]\), and where we used the multi-index notation \(y^{(r+1)I}=(y_0^{i_0}\cdots y_n^{i_n})^{r+1}\) for \(I=(i_0,\dots,i _n)\) and homogeneous coordinates \([y_0,\dots, y_n]\) on \(\P^n\). Set \(m:=\ep+(r+1)\delta\). For any \(i\in \{1,\dots, n\}\) we obtain, as explained in \cref{meromorphic connection}, a logarithmic connection \(\nabla_{F_i}:\O_{\P^n}(m)\to \O_{\P^n}(m)\otimes \Omega_{\P^n}(\log H_i)\) defined by 
\[\nabla_{F_i}(s)\stackrel{\rm loc}{:=}ds-s\frac{d F_i}{F_i}.\]
This connection satisfies moreover trivially the  relation  \(\nabla_{F_i}(F_i)\equiv 0\).
A straightforward computation therefore implies that
\[0=\nabla_{F_i}(F_i)=\sum_{|I|=\delta_i}\alpha_{i,I}y^{rI}=\sum_{|I|=\delta_i}\alpha_{i,I}x^{I}\]
where for each \(I\), \(\alpha_{i,I}\in H^0(\P^n,\Omega_{\P^n}(\log H_i)\otimes \O_{\P^n}(\ep+\delta))\) and where \((x_0,\dots, x_n)=(y_0^r,\dots,y_n^r)\). From this observation, one can see that this collection of connexions induces a rational map \[\P(\Omega_{\P^n}(\log D))\dashrightarrow \ys\]
where \(\ys\) is the universal family  of complete intersections of codimension \(n\) and multidegree \((\delta,\dots, \delta)\):
\[\ys:=\{(P_1,\dots,P_n,[x])\in |\O_{\P^n}(\delta)|^{\times n}\times \P^n\ | \ P_1(x)=\cdots=P_n(x)=0\}.\]
If one denotes  \(\O(b_1,\dots, b_n):=\O_{|\O_{\P^n}(\delta)|}(b_1)\times \cdots\times\O_{|\O_{\P^n}(\delta)|}(b_n) \) for any integers \(b_1,\dots, b_n\), then, as we shall see, at least for \(r\) large enough, every element in 
\[H^0\big(\ys,\O(b_1,\dots, b_n)\boxtimes \O_{\P^n}(-1)|_{\ys}\big)\]
induces a symmetric logarithmic differential form of \((\P^n,D)\) vanishing along some ample divisor.  The map \(\ys\to |\O_{\P^n}(\delta)|^{\times n}\) being generically finite, if the \(b_i\)'s are positive, the pull back of \(\O(b_1,\dots, b_n)\) to \(\ys\) is big and nef. Therefore, at least when the \(b_i\)'s are large enough, there are many  global sections of \(\O(b_1,\dots, b_n)\boxtimes \O_{\P^n}(-1)|_{\ys}\).   Moreover, the base locus in \(\ys\) of these global section can be understood geometrical from a result of Nakamaye \cite{Nak00} and one can control the dimension of this base locus  in view of  a   work of Benoist \cite{Ben11}. 
One can thus  hope to use this information to  understand the augmented base locus of the tautological line bundle on \(\P(\Omega_{\P^n}(\log D))\).

Unfortunately, there are several technical difficulties which make the proof rather involved. First of all, as mentioned above, we first have to take a suitable modification of \(\P(\Omega_{\P^n}(\log D))\), this relies on the most technical part of our argument, namely an explicit resolution of the obstructions induced by the residues.  Even so, it is rather delicate to understand the map from \(\P(\Omega_{\P^n}(\log D))\) to \(\ys\), mainly because the dependence of the \(\alpha_{i,I}\)'s on the \(a_{i,I}\)'s is non-linear. Therefore we had to modify the above argument by constructing an embedding of the pair \((\P^n,D)\stackrel{\sim}{\to}(Z,D'|_{Z})\subset (X,D')\) where now \(Z\) depends on the \(a_{i,I}\)'s, but the pair \((X,D')\) does not.  Moreover, because of our particular shape of equation, there are also  complications occurring along the coordinate hyperplanes. These seem to be unavoidable with this approach, as in \cite{BD15,Bro16,Bro17,Den17,Xie15}.  

Lastly, let us mention that, although in the case \((\pb^n,\sum_{i=1}^{c}H_i) \), the condition \(c\geqslant n\) in \cref{cor:simple} is optimal, there exists many examples  with fewer components having almost ample logarithmic cotangent bundles. Such examples as well as generalizations of positivity to higher order jet spaces will be studied in a forthcoming work.\\

The organization of the paper is as follows. In \cref{se:LogCotangent}, we make some preliminary observation on the geometry of the projectivization of the logarithmic cotangent bundle and study  the ampleness  obstructions  induced by the residue maps. We moreover explain how these obstructions are also indeterminacy loci of natural rational maps between the projectivizations of the logarithmic cotangent bundles of \((X,D)\) and \((X,D')\) where \(D'\) is the sum of only some of the components of \(D\).  In \cref{se:Main} we prove our main result by implementing the strategy explained above. In \cref{apprendix a} we prove that our result is optimal on the number of components by establishing a vanishing theorem for logarithmic symmetric differential form on \((\P^n,D)\) where \(D\) is a divisor with less than \(n\) components. Lastely, in section \cref{se:Resolution}, we describe an explicit resolution algorithm for the  indeterminacies between the different rational maps studied in \cref{se:LogCotangent}.



\section{The geometry of the projectivized logarithmic cotangent bundle}\label{se:LogCotangent}
\subsection{Conventions} 
We summarize here the main conventions we use throughout this text. We work over the field of complex numbers \(\cb\).   Given a complex manifold \(X\) and a line bundle \(L\) on \(X\),  for any global section \(\sigma\) of \(L\),   we denote by \((\sigma=0)\subset X\) the zero locus of \(\sigma\). 
The base locus \({\rm Bs}(L)\) is the intersection of the zero loci of all global sections of \(L\) and \(\mathbf{B}(L):=\bigcap_{m\in \nb} {\rm Bs}(L^m)\).  If \(X\) if projective,  then the augmented base locus of \(L\) is 
\[\mathbf{B}_{+}(L):=\bigcap_{m\in \nb}\mathbf{B}(L^m\otimes A^{-1})\]
for any ample line bundle  \(A\) on \(X\) (see \cite{Laz04II}). Given a vector bundle \(E\) on \(X\) we will denote by \(\P(E)\) the projectivization of rank one quotients of \(E\) and by \(\O_{\P(E)}(1)\) the tautological line bundle on \(\P(E)\). We will often identify \(\P(E)\) with \(P(E^\vee)\) the projectivisation of lines in \(E^\vee\) and thus denote the elements in \(\P(E)\) in the form \([\xi]\) with \(\xi\in E^{\vee}_x\) for some \(x\in X\). Given a divisor \(D\) on \(X\) and a subvariety \(Z\subset X\), we denote by \(D|_Z\) the divisor \(\iota^*D\) where \(\iota:Z\to X\) is the inclusion morphism. Given a morphism \(\rho:\X\to S\) we will denote by \(X_s=\rho^{-1}(\{s\})\) the fiber above  \(s\in S\).

\subsection{Log pairs and logarithmic cotangent bundles}\label{def:log pair}
Let $X$ be a (not necessarily compact) complex manifold of dimension \(n\geqslant 1\). Let $D=\sum_{i=1}^{c}D_i$ be a simple normal crossing divisor on $X$, that is, for any \(x\in X\), there exists an open neighborhood \(U\) of \(x\) with coordinates \((z_1,\dots, z_n)\) centered at \(x\) and an integer \(k\leqslant c\) such that 
\[D\cap U=(z_1\dots z_k=0).\]
Such a pair $(X,D)$ will be called a  \emph{smooth pair} or a \emph{log pair}.  One denotes by $T_X(-\log D)$
the logarithmic
tangent bundle of $(X,D)$. Recall that this  is the locally free subsheaf of the holomorphic
tangent bundle $T_X$ of  vector fields tangent to \(D\). Locally on an open subset of $U\subset X$ with local coordinates $(z_1,\ldots,z_n)$ as above, $T_X(-\log D)$ is generated by
$$
\left(z_1\frac{\d}{\d z_1},\ldots,z_k\frac{\d}{\d z_k}, \frac{\d}{\d z_{k+1}}, \ldots, \frac{\d}{\d z_n}\right).
$$ 
Consider the dual of $T_X(-\log D)$, which is the locally free sheaf generated by
$$
\left(\frac{d z_1}{z_1},\ldots,\frac{d z_k}{z_k},  d z_{k+1}, \ldots,  d z_n\right).
$$
It is  denoted by $\Omega_X(\log D)$ and is called the \emph{logarithmic cotangent bundle} of $(X,D)$. Note that $\Omega_X$ is a subsheaf of $\Omega_X(\log D)$.

The projectivization \(\P(\Omega_X(\log D))\) of \(\Omega_X(\log D)\) will be critical in this paper, therefore we introduce the following notation 
\[X_1(D):= \P(\Omega_X(\log D)).\]
We will also denote by \(\pi_{X,D}:X_1(D)\to X\) the canonical projection, but we will mostly write \(\pi_X\) instead of \(\pi_{X,D}\) for readability. This slight abuse of notation should not lead to any confusion.


Given another log pair $(Y,E)$, and a morphism \(f:X\to Y\), we say that \(f\) is a morphism of log pairs if  $f^{-1}(E)\subset D$, and in this case, we write $f:(X,D)\rightarrow (Y,E)$. The differential of such a morphism induces a morphism
\[df:TX(-\log D)\to f^*T_Y(-\log E),\]
or dually, a morphism 
\[{}^tdf:f^*\Omega_Y(\log E)\to \Omega_X(\log D).\]
Therefore, one obtains a rational map 
\[[df]:X_1(D)\dashrightarrow Y_1(E).\]
A particular instance of a morphism of log pairs is given when one considers 
a smooth submanifold $Z$ of $X$. If $Z$ intersects  $D$ transversely (i.e. \(D|_Z\) is simple normal crossing), then  $(Z,D|_Z)$ is also a log manifold and  the inclusion morphism $i_Z:(Z,D|_Z)\rightarrow (X,D)$ is a   morphism of log pairs. Moreover, the induced meromorphic map \([d{i}_Z]: Z_1(D|_Z)\rightarrow X_1(D)\) is holomorphic.

\subsection{Obstruction to ampleness}\label{sse:obstruction} Take \((X,D)\) as above. For any \(i\in \{1,\dots,c\}\), there exists a residue map
\[\Res_{D_i}:\Omega_X(\log D)\to \O_{D_i}\]
defined, over any open subset \(U\subset X\), by 
\[\Res_{D_i}\left(\alpha+\sum_{k=1}^c\beta_k\frac{d\sigma_k}{\sigma_k}\right)=\beta_i|_{D_i},\] for any \(\alpha\in \Gamma(U,\Omega_X)\) any \(\beta_1,\dots, \beta_c\in \O_X(U)\) and where \(D_k=(\sigma_k=0)\) for any \(k\in \{1,\dots,c\}\). This map is easily seen to be a well defined morphism of \(\O_X\)-modules.

For any non-empty $I=\{i_1,\ldots,i_r\}\subset \{1,\ldots,c\}$, set $I^{\complement}:=\{1,\ldots,c\}\setminus I$,    define $D_I:=D_{i_1}\cap\cdots\cap D_{i_r} $ which is a smooth complete intersection of dimension \(n-r\) since the divisor \(D\) is simple normal crossing, and \(D(I^{\complement})=\sum_{i\in I^{\complement}}D_i\) which is a simple normal crossing divisor. We obtain the following short exact sequence of sheaves
\begin{equation}\label{eq:Residues}
0\rightarrow \Omega_X\big(\log D(I^\complement)\big)\rightarrow \Omega_X(\log D)\xrightarrow{{\rm Res}} \oplus_{i\in I}\oc_{D_i}\rightarrow 0,
\end{equation}
which induces a quotient of vector bundles over the (not necessarily  connected) smooth   submanifold $D_{I}\subset X$ of codimension $r$  in \(X\):
\begin{equation}\label{eq:QuotientObstruction}
\Omega_X(\log D)\mid_{D_{I}}\rightarrow \oc_{D_{I}}^{\oplus r}\rightarrow 0.
\end{equation}
Since this is a morphism of \(\O_{D_I}\)-modules, it induces an inclusion   
\[\tilde{D}_I:=\pb(\oc_{D_I}^{\oplus r})\cong D_{I}\times \pb^{r-1}\subset \pb\big(\Omega_X(\log D)\big).\] 
 Observe that \(\dim \tilde{D}_I=\dim D_I+\# I-1=n-1\) for any \(I\subset \{1,\dots,c\}\).
From \eqref{eq:QuotientObstruction}, we see that when \(X\) is projective of dimension \(n\geqslant 2\) and   the number of components of \(D\) is positive (i.e. \(D\neq \varnothing\)),  the logarithmic cotangent bundle \(\Omega_X(\log D)\) cannot be ample. Indeed,  for any non-empty \(I\subset \{1,\dots, c\}\) with \(\#I<n\), the  restriction of \(\Omega_X(\log D)\) to \(D_I\) has a trivial quotient, which, since \(\dim D_I=n-\#I>0\) prevents the logarithmic cotangent bundle from being ample. Said differently, it implies that \(\tilde{D}_I\subset \mathbf{B}_+\big(\O_{X_1(D)}(1)\big)\) for any such \(I\).  Observe that when \(\# I>n\) then \(D_I\) is empty, and that when \(\#I=n\) then \(D_I\) is a finite set of point, therefore, in this last case it does not directly follow that \(\tilde{D}_I\subset \mathbf{B}_+(\O_{X_1(D)}(1))\).
This leads us to introduce the following definition.
\begin{dfn}\label{almost ample}
	 Let $D=\sum_{i=1}^{c}D_i$ be a simple normal crossing divisor on a projective manifold $X$ of dimension \(n\). We say that the log pair $(X,D)$ has \emph{almost ample logarithmic cotangent bundle} if the tautological line bundle $\oc_{X_1(D)}(1)$ is big and  its augmented base locus satisfies
	$$
	\mathbf{B}_+\big(\oc_{X_1(D)}(1)\big)= \bigcup_{\substack{\emptyset\subsetneq I\subset \{1,\ldots,c\}\\
	\#I< n}}\tilde{D}_I.
	$$ 
\end{dfn}

The situation can be understood in local coordinates as follows. Fix a point \(x\in X\). Without loss of generality, we can suppose that there exists \(k\leqslant c\) such that \(\{i\in \{1,\dots,c\}\ |\ x\in D_i\}=\{1,\dots,k\}\). One  can therefore take  an open neighborhood $U$ of $x$ with local coordinates $(z_1,\ldots,z_n)$ such that     $D_i\cap U=(z_i=0)$ for any $i\in\{1,\ldots,k\}$. In this setting, one has a local trivialization
\begin{eqnarray*}
 U\times \pb^{n-1}&\stackrel{\sim}{\longrightarrow}&{\pi}_{X}^{-1}(U)\\
(z,[\xi_1,\ldots,\xi_n])&\mapsto&	\left(z,\Big[\sum_{i=1}^k\xi_iz_i\frac{\d}{\d z_i}+\sum_{i=k+1}^n\xi_i\frac{\d}{\d z_i}\Big]\right).
\end{eqnarray*} 

Take \(I\subset\{1,\dots,c\}\). Observe that \(D_I\cap U=\varnothing\) if \(I\not\subset \{1,\dots,k\}\). If on the other hand \(I=\{i_1,\dots,i_r\}\subset \{1,\dots,k\}\), then from the 
definition of the residue map, we see that in these local coordinates, the restricted residue map \(\Res:\Omega_X(\log D)|_{D_I\cap U}\to \O_{D_I\cap U}^{\oplus r}\) is given by
\[\Res\left(\sum_{i=1}^k\eta_{i}\frac{dz_i}{z_i}+\sum_{i=k+1}^n\eta_idz_i\right) = (\eta_{i_1},\dots, \eta_{i_r}).\]
Therefore, it follows that under the above trivialization,  \(\tilde{D}_I\) is given by
\begin{eqnarray}\label{local}
\tilde{D}_{I}\cap{\pi}_{X}^{-1}(U)=\big\{(z,[\xi_1,\ldots,\xi_n])\in U\times \P^{n-1} \ |\   z_i=0\ \forall i\in   I\ \ \text{and}\ \   \xi_j=0  \ \forall j\in\{1,\dots,n\}\setminus  I\big\}.
\end{eqnarray}
For later use, let us observe that this implies that for any $\varnothing\neq I,J\subset \{1,\ldots,c\}$, since \((I\cap J)^\complement=I^\complement \cup J^\complement\),
\begin{eqnarray}\label{intersection}
\begin{cases}
 \tilde{D}_I\cap \tilde{D}_J\subset   \tilde{D}_{I\cap J}   & \quad \text{if } I\cap J\neq\varnothing\\
\tilde{D}_I\cap \tilde{D}_J=\varnothing  & \quad \text{if } I\cap J=\varnothing.
\end{cases}
\end{eqnarray}
\begin{dfn}\label{obstruction} For a non-empty subset \(I\subset \{1,\dots, c\}\), we define 
	 $\jc_I$ to be the ideal sheaf of $\tilde{D}_I$ in $X_1(D)$ and  $\js_I:=\bigcap_{{J\subset I }}\jc_J$. 
	  We also write   $\js_D:= \bigcap_{{\varnothing\neq J }}\jc_J$ and $\js_i:=\js_{\{i\}^\complement}$ for brevity.
\end{dfn} 
For any \(I\subset \{1,\dots, c\}\), the map 
\[g_I:\Omega_X\big(\log D(I^\complement)\big)\to \Omega_X(\log D)\]
induces a rational map
\begin{equation}\label{rational}
\gamma_I:X_1(D)\dashrightarrow X_1\big( D(I^\complement)\big).\end{equation}
If one denotes by \(U(\gamma_I)\subset X_1(D)\) the complement of the indeterminacy locus of \(\gamma_I\), then one has an isomorphism 
\begin{equation}\label{eq:PullBackIso}
 \O_{X_1(D)}(1)|_{U(\gamma_I)}\stackrel{\sim}{\to}\gamma_I^*(\O_{X_1(D(I^\complement))}(1)).
\end{equation}

On the other hand, for any line bundle \(L\) on \(X\), one has isomorphisms 
\begin{eqnarray*}
H^0(X,\Omega_X(\log D)\otimes L)&\cong& H^0(X_1(D),\O_{X_1(D)}(1)\otimes \pi_X^*L)\\
H^0\big(X,\Omega_X\big(\log D(I^\complement)\big)\otimes L\big)&\cong& H^0\big(X_1\big(D(I^{\complement})\big),\O_{X_1(D(I^{\complement}))}(1)\otimes \pi_X^*L\big).
\end{eqnarray*}
as we also have a  map \(H^0\big(X,\Omega_X\big(\log D(I^\complement)\big)\otimes L\big)\to H^0(X,\Omega_X(\log D)\otimes L)\) induced by \(g_I\), we obtain a map
\[\gamma_I^*:H^0\big(X_1\big(D(I^{\complement})\big),\O_{X_1(D(I^{\complement}))}(1)\otimes \pi_X^*L\big)\to H^0(X_1(D),\O_{X_1(D)}(1)\otimes \pi_X^*L).\]
The choice of notation is legitimate because on \(U(\gamma_I)\) it coincides with the map induced by \eqref{eq:PullBackIso}.
\begin{proposition}\label{resolve} With the same notation as above, take \(I\subset \{1,\dots,c\}\).
\begin{thmlist}
\item \label{equivalent simultaneous}Let  $\tilde{\mu}:\widetilde{X}_1\rightarrow X_1(D)$ be a log resolution of the ideal sheaves $\js_I$. Then $\tilde{\mu}$ resolves the indeterminacy of the meromorphic map \(\gamma_I\).
\item Given any line bundle \(L\) on \(X\), and any element \(\sigma\in H^0\big(X_1\big(D(I^{\complement})\big),\O_{X_1(D(I^{\complement}))}(1)\otimes \pi_X^*L\big)\), the section \(\gamma_I^*\sigma\) vanishes along \(\js_I\). Equivalently, \(\gamma_I^*\) factors through
\[H^0\big(X_1\big(D(I^{\complement})\big),\O_{X_1(D(I^{\complement}))}(1)\otimes \pi_X^*L\big)\to H^0(X_1(D),\O_{X_1(D)}(1)\otimes \pi_X^*L\otimes \js_I).\]
\item If \(W\subset H^0\big(X_1\big(D(I^{\complement})\big),\O_{X_1(D(I^{\complement}))}(1)\otimes \pi_X^*L\big)\) is a base point free linear system, then the ideal sheaf defined by  \(\gamma_I^*W\) is precisely \(\js_I\).
\end{thmlist}

\end{proposition}
\begin{proof}
	Without loss of generality, one can assume that $I=\{1,\ldots,r\}$. 
	Take  $x\in X$. One can assume that there exists an open neighborhood $U$ of $x$ with local coordinates $(z_1,\ldots,z_n)$, and  integers $m\in \{0,\dots, r\} $, \(q\in \{r,\dots, c\}\),   such that
	 \[U\cap D=U\cap (D_1\cup\cdots \cup D_m\cup D_{q+1}\cup\cdots \cup D_c)\]  
	 and such that  $D_i\cap U=(z_i=0)$ for any $i\in\{1,\dots,m,q+1,\dots,c\}$, here we use the convention \(D_1\cup\cdots \cup D_m=\varnothing\) if \(m=0\). In this setting, both $\pb\big(\Omega_X(\log D)|_U\big)$ and $\pb\big(\Omega_X\big(\log D(I^\complement)\big)|_U\big)$ are isomorphic to the trivial product $U\times \pb^{n-1}$ in such a way that one has the following commutative diagram
	\[
	\xymatrix{
	\pb\big(\Omega_X(\log D)|_U\big)\ar[d]_{\cong}\ar@{-->}[r]^-{\gamma_I} & \pb\big(\Omega_X\big(\log D(I^\complement)\big)|_U\big)\ar[d]_{\cong}\\
	U\times \pb^{n-1}\ar@{-->}[r]^{f_U}& U\times \pb^{n-1}
}
	\]
	where \[f_U(z,[\xi_1,\ldots,\xi_n])= (z,[z_1\xi_1,\ldots,z_m\xi_m,\xi_{m+1},\ldots,\xi_n]).\] 
	On the other hand, for any \(J=\{j_1,\dots, j_s\}\subset \{1,\dots, r\}=I\), writing \(\{k_{s+1},\dots,k_n\}=\{1,\dots, n\}\setminus J\) one has 
	\[\jc_J\big(\pi_X^{-1}(U)\big)=(z_{j_1},\dots, z_{j_s},\xi_{k_{s+1}},\dots, \xi_{k_n}).\]
	But then, as we shall see in \Cref{sec:simple sheaves}, it follows  that
\begin{eqnarray}\label{eq:lift ideal}
	\js_I\big(\pi_X^{-1}(U)\big)=(z_1\xi_1,\dots,z_m\xi_m,\xi_{m+1},\dots, \xi_n).
\end{eqnarray}
From the expression of \(f_U\) it is clear that any resolution of \(\js_I\) resolves the indeterminacies of \(f_U\), and therefore also of \(\gamma_I\). This proves the first point. 
	
The verification of the second point can be done locally. Take a line bundle \(L\) and a global section \(\sigma\in H^0\big(X_1\big(D(I^{\complement})\big),\O_{X_1(D(I^{\complement}))}(1)\otimes \pi_X^*L\big)\). We can suppose that \(L|_U\) is trivialized. Above the open subset \(U\), the restriction \(\sigma_U:=\sigma|_{\pi_X^{-1}(U)}\) is of the form
\[\sigma_U=\sum_{i=1}^ns_i\xi_i,\]
	for some \(s_i\in \O ( U )\). Then in our choice of coordinates \[(\gamma_i^*\sigma)|_{\pi_X^{-1}(U)}=\sum_{i=1}^ms_iz_i\xi_i+\sum_{i=m+1}^ns_i\xi_i.\]
	In view of the above description for \(\js_I\), it is immediate that \(\gamma_i^*\sigma|_{\pi_X^{-1}(U)}\) vanishes along \(\js_I(\pi_X^{-1}(U))\), and since this holds for any \(x\in X\) and any small enough neighborhood \(U\) of \(x\), the second assertion follows at once. The third assertion is also an immediate consequence of the local  expression for \((\gamma_i^*\sigma)|_{\pi_X^{-1}(U)}\).
%
\end{proof}
Let us  now show that at least under some additional assumption, the  ``almost ampleness'' property is preserved if one add more components. 
\begin{proposition}\label{cor:MoreComponents}
Let  \(X\) be a smooth projective variety of dimension \(n\). Take an integer \(c>n\) and a simple normal crossing divisor \(D=\sum_{i=1}^cD_i\) on \(X\). For any \(j\in \{1,\dots,c\}\) let us define \(D_{\hat{\jmath}}:=D-D_j\). If for any \(j\in \{1,\dots, n+1\}\) the pair \((X,D_{\hat{\jmath}})\) has almost ample logarithmic cotangent bundle, then the pair \((X,D)\) has almost ample logarithmic cotangent bundle.

In particular, if \((X,D_{\hat{\jmath}})\) has almost ample logarithmic cotangent bundle for any \(j\in \{1,\dots,c\}\) then so does the pair \((X,D)\).
\end{proposition}
\begin{proof}
For any \(j\in \{1,\dots, c\}\), set \(U_j:=X\setminus D_{j}\).  Since \(D\) is simple normal crossing, one has \(X=\bigcup_{j=1}^{n+1}U_j\) and therefore it suffices to prove that  for any \(j\in \{1,\dots, n+1\}\), one has
\begin{equation}\label{eq:GoalB_+}
\mathbf{B}_+(\O_{X_1(D)}(1))\cap\pi_{X}^{-1}(U_j)=\bigcup_{\substack{I\subset\{1,\dots,c\}\\ \# I<n}}\tilde{D}_I\cap\pi_X^{-1}(U_j).
\end{equation}
Without loss of generality, since our assumption is on all \(j\in \{1,\dots,n+1\}\), it suffices to prove \eqref{eq:GoalB_+} for \(j=1\). 
Observe that for any \(I\in \{1,\dots, c\}\), if \(1\in I\), then \(\tilde{D}_I\cap \pi_X^{-1}(U_1)=\varnothing\). The right hand side of \eqref{eq:GoalB_+} is therefore 
\[\bigcup_{\substack{I\subset\{2,\dots,c\}\\ \# I<n}}\tilde{D}_I\cap\pi_X^{-1}(U_1).\]
We now resolve the ideal sheaf \(\is_{\{1\}}=\jc_{\{1\}}\subset \O_{X_1(D)}\) by blowing up \(\tilde{D}_{\{1\}}\). Write \(D'=D-D_1\) and \(\widetilde{X}_1:={\rm Bl}_{X_1(D)}(\tilde{D}_{\{1\}})\stackrel{\tilde{\mu}}{\to}X_1(D)\). By \cref{resolve}, the map \(\tilde{\mu}\) resolves the indeterminacies of \(\gamma_{\{1\}}\) and therefore we have a commutative diagram 
\[
\xymatrix{
\widetilde{X}_1\ar[d]_{\tilde{\mu}}\ar[dr]^{\nu}&\\
X_1(D)\ar[dr]_{\pi_{X,D}}\ar@{-->}[r]_{\gamma_{\{1\}}}&X_1(D')\ar[d]^{\pi_{X,D'}}\\
& X\\
}
\]
satisfying moreover
\begin{equation}\label{eq:RelationPullBack}
\nu^*\O_{X_1(D')}(1)=\tilde{\mu}^*\O_{X_1(D)}(1)\otimes \O_{\widetilde{X}_1}(-E)\end{equation}
where \(E\) is the exceptional divisor of \(\tilde{\mu}\).

By hypothesis, \((X,D')\) has almost ample logarithmic cotangent bundle, therefore 
\[\mathbf{B}_+(\O_{X_1(D')}(1))=\bigcup_{\substack{I\subset \{2,\dots,c\}\\ \# I<n}}\tilde{D}'_I.\]
One can find  \(m\in \nb\) and an ample line bundle \(A\) on \(X\) such that \(\mathbf{B}_+(\O_{X_1(D')}(1))={\rm Bs}\big(\O_{X_1(D')}(m)\otimes \pi_{X,D'}^*A^{-1}\big)\).
After pulling back to \(\widetilde{X}_1\) and in view of \eqref{eq:RelationPullBack} this implies
\[{\rm Bs}\big(\tilde{\mu}^*(\O_{X_1(D)}(m)\otimes \pi_{X,D}^*A^{-1})\big)\subset {\rm Bs}\big(\nu^*(\O_{X_1(D')}(m)\otimes \pi_{X,D'}^*A^{-1})\big)\cup\Supp (E)\subset \nu^{-1}\Big(\mathbf{B}_+\big(\O_{X_1(D')}(1)\big)\Big)\cup \Supp(E).\]
Since \(\tilde{\mu}\) is birational and \(X_1(D)\) is smooth we obtain
\[{\rm Bs}\big(\O_{X_1(D)}(m)\otimes \pi_{X,D}^*A^{-1}\big)\subset \tilde{\mu}\Big(\nu^{-1}\Big(\mathbf{B}_+\big(\O_{X_1(D')}(1)\big)\Big)\cup \Supp(E)\Big)\subset {\gamma_{\{1\}}^{-1}\big(\mathbf{B}_+(\O_{X_1(D')}(1))\big)}\cup \tilde{D}_1,\] 
where \(\gamma_{\{1\}}\) is understood to be restricted to \(X_1(D)\setminus \tilde{D}_1\), where it is regular. Therefore 
\begin{equation}\label{eq:PresqueFini}
\mathbf{B}_+(\O_{X_1(D)}(1))\subset {\gamma_{\{1\}}^{-1}\big(\mathbf{B}_+(\O_{X_1(D')}(1))\big)}\cup \tilde{D}_1\subset \bigcup_{\substack{I\subset \{2,\dots,c\}\\ \# I<n}}\tilde{D}_I\cup \pi_X^{-1}(D_1).\end{equation}
To deduce the last inclusion, we use the fact that outside \(\pi_X^{-1}(D_1)\) the map \(\gamma_{\{1\}}\) is an isomorphism and that for any \(I\subset \{2,\dots,c\}\) with \(\# I<n\) one has \(\gamma_{\{1\}}^{-1}(\tilde{D}'_I)\cap \pi_{X,D}^{-1}(U_1)=\tilde{D}_I\cap \pi_{X,D}^{-1}(U_1)\). It the suffices to take the  intersection with \(\pi_X^{-1}(U_1)\) on both sides of \eqref{eq:PresqueFini} to deduce \eqref{eq:GoalB_+}.

\end{proof}
If in  \cref{cor:MoreComponents} we do not ask a condition on the \(D_{\hat{\jmath}}\) for all \(j\in \{1,\dots, n+1\}\), but only on one of them, then the above argument does not work. However, if one slightly weakens the expected conclusion by allowing the augmented base locus of \(\O_{X_1(D)}(1)\) to contain the different \(\tilde{D}_I\) for \(\# I=n\) as well, then one has the following stronger result.
\begin{proposition}
Let \(X\) be a smooth projective variety of dimension \(n\). Take \(c\in \nb^*\) and a simple normal crossing divisor  \(D=\sum_{i=1}^cD_i\) on \(X\). Let \(D':=\sum_{i=2}^cD_i\). If 
\[\mathbf{B}_+(\O_{X_1(D')}(1))=\bigcup_{\substack{I\subset \{1,\dots, c\}\\ \# I\leqslant n}}\tilde{D}'_I,\]
then 
\(\displaystyle{\mathbf{B}_+(\O_{X_1(D)}(1))=\bigcup_{\substack{I\subset \{1,\dots, c\}\\ \# I\leqslant n}}\tilde{D}_I}.\)
\begin{proof} As in the prof of \cref{cor:MoreComponents}, one can prove that 
\[\mathbf{B}_+(\O_{X_1(D)}(1))\subset \gamma_{\{1\}}^{-1}\big(\mathbf{B}_+(\O_{X_1(D')}(1))\big)\cup \tilde{D}_1.\]
Therefore we are just reduced to understand \(\gamma_{\{1\}}^{-1}\big(\mathbf{B}_+(\O_{X_1(D')}(1))\big)\). This can be done locally. Outside \(D_1\) there is nothing to prove because \(\gamma_{\{1\}}\) is an isomorphism and that moreover for any \(I\) one has \(\gamma_{\{1\}}^{-1}(\tilde{D}'_I)\cap \pi_{X,D}^{-1}(U_1)=\tilde{D}_I\cap \pi_{X,D}^{-1}(U_1)\), where \(U_1=X\setminus D_1\). Let \(x\in D_1\) and take an open neighborhood \(U\) of \(x\) with coordinates \((z_1,\dots,z_n)\) centered at \(x\) such that there exists \(k\in \{1,\dots,n\}\) for which \(D\cap U=(z_1\cdots z_k=0)\) and \(D_i\cap U=(z_i=0)\) for all \(i\in \{1,\dots, k\}\). Using the trivialization described above, the map \(\gamma_{\{1\}}\) is then given by
\[\gamma_{\{1\}}(z_1,\dots,z_n,[\xi_1,\dots,\xi_n])=(z_1,\dots,z_n,[z_1\xi_1,\xi_2,\dots,\xi_n]).\]
On the other hand, for any non-empty \(I=\{i_1,\dots, i_r\}\subset \{2,\dots,c\}\) with \(\# I=r\leqslant n\), writing \(\{j_1,\dots j_{n-r}\}=\{1,\dots,n\}\setminus I\), the equations defining \(\tilde{D}'_I\) in the  above coordinates for \(X_1(D')\) are
\[\tilde{D}'_I=(z_{i_1},\dots,z_{i_r},\xi_{j_1},\dots,\xi_{j_{n-r}}=0).\]
Since \(1\notin I\), one has \(1\in J\), and we can therefore suppose \(j_1=1\). Therefore, \(\gamma_{\{1\}}^{-1}(\tilde{D}'_I)\) is given in the above coordinates by 
\[(z_{i_1},\dots,z_{i_r},\xi_1,\xi_{j_2},\dots,\xi_{j_{n-r}}=0)\cup(z_{i_1},\dots,z_{i_r},z_1,\xi_{j_2},\dots,\xi_{j_{n-r}}=0)=\tilde{D}_I\cup \tilde{D}_{I\cup\{1\}}.\]
Since  \(\tilde{D}_{I\cup\{1\}}=\varnothing \) whenever \(\# I=n\), the result follows.
\end{proof}
\end{proposition}
\subsection{Families of smooth pairs}
Let us collect here some elementary observations concerning families of smooth pairs.
\begin{dfn}\label{def:smooth pair}
	A \emph{family of smooth pairs} consists in the following data:
\begin{enumerate}
\item A smooth quasi-projective variety \(\X\), a smooth quasi-projective variety \(S\) and a smooth proper morphism 
\[\rho:\X\to S.\]
\item A simple normal crossing divisor \(\mathscr{D}=\sum_{i=1}^c\mathscr{D}_i\) on \(\X\) such that given any \(s\in S\), denoting by \(\iota_s:X_s\hookrightarrow \X\) the canonical injection, the divisor \(D_s:=\iota_s^*\mathscr{D}\) is simple normal crossing.
\end{enumerate}
We will denote such a family by \((X_s,D_s)_{s\in S}\) or more precisely by \((\X,\mathscr{D})\stackrel{\rho}{\to}S\) if needed.
\end{dfn}
Let us first  observe that as a consequence of the local inverse theorem in several complex variables we obtain that in the analytic category, families of smooth pairs are locally trivial.
\begin{lem}\label{lem:LocTrivial}
Let \((\X,\ds)\stackrel{\rho}{\to} S\) be a family of smooth pairs. Set \(n:=\dim \rho\). For any \(x\in \X\) there exists a neighborhood \(U\subset \X\),  a neighborhood \(U_1\subset S\) of \(\rho(x)\) and an open subset \(U_2\subset \cb^n\) with coordinates \((z_1,\dots,z_n)\) such that there exists an isomorphism 
\[\Phi:U_1\times U_2\to U \]
satisfying \(\rho\circ \Phi=\pr_1\) (where \(\pr_1:U_1\times U_2\to U_1\) is the projection on the first factor) and 
such that for some \(k\leqslant  n\) one has
\[\Phi^*\ds=(z_1 \cdots z_k=0).\]
\end{lem}
Let us now make the following observation.
\begin{lem}\label{lem:AlgebraicSmooth} Let \(\X\) be  a smooth quasi-projective variety endowed  with a simple normal crossing divisor \(\mathscr{D}\). Let \(S\) be a quasi-projective variety and suppose that we are given a proper morphism \(\rho:\X\to S\). If there exists \(s_0\in S\) such that \((X_{s_0},D_{s_0})\) is a smooth pair (where \(D_{s_0}=\mathscr{D}|_{X_{s_0}}\)), then there exists a non-empty Zariski open subset \(U\subset S\) such that the restricted family \((\rho^{-1}(U),\mathscr{D}|_{\rho^{-1}(U)})\stackrel{\rho}{\to} U\) is a family of smooth pairs.
\end{lem}
\begin{proof} This is a consequence of the properness of the map \(\rho\) since the locus of point \(x\in \X\) at which the divisor \(D_{\rho(x)}\) is not simple normal crossing is closed.
\end{proof}

To any family of smooth pairs \((\X,\ds)\stackrel{\rho}{\to}S\),  one can associate a relative logarithmic cotangent bundle \(\Omega_{\X/S}(\log \ds)\) defined as the quotient of \(\Omega_{\X}(\log \ds)\) by \(\rho^*\Omega_S\). By definition, it sits in the following commutative diagram 
	\[
	\xymatrix{
0\ar[r] & \rho^*\Omega_{S}\ar[r]\ar@{=}[d]&\Omega_{\X}\ar[r]\ar[d] &\Omega_{\X/S}\ar[r]\ar[d]& 0 \\
0\ar[r] & \rho^*\Omega_{S}\ar[r]&\Omega_{\X}(\log \ds)\ar[r]&\Omega_{\X/S}(\log \ds)\ar[r]& 0. 
}
	\]
	Observe that Lemma \ref{lem:LocTrivial} implies that  \(\Omega_{\X/S}(\log \ds)\) is a locally free sheaf. We denote its dual by \(T_{\X/S}(-\log \ds)\). Observe also that for any \(s\in S\), one has \(\Omega_{X_s}(\log D_s)=\Omega_{\X/S}(\log \ds)|_{X_s}\) and \(T_{X_s}(-\log D_s)=T_{\X/S}(-\log \ds)|_{X_s}\).
\subsection{A resolution algorithm}
Let us state here the main properties of a resolution algorithm we will construct in  \Cref{se:Resolution}.
\begin{proposition}\label{prop:ResolutionAlgo}
There exists a resolution algorithm which to every smooth pair \((X,D)\) such that \(D=\sum_{i=1}^cD_i\) with \(c\leqslant \dim X\), associates a  smooth variety \(\widehat{X}_1(D)\) and a birational morphism (the so-called \emph{minimal resolution} in \Cref{indeterminacies})
\[\mu:\widehat{X}_1(D)\to X_1(D),\]
satisfying the following properties.
\begin{thmlist}
\item \label{part1}With the notation of Definition \ref{obstruction}, the morphism \(\mu\) is a resolution of \(\js_I\) for every \(I\subsetneq \{1,\dots, c\}\). In particular \(\mu\) is a simultaneous log resolution of \(\{\js_i\}_{i\in\{1,\ldots,c\}}\). Moreover, \(\mu\) is birational outside \(\bigcup_{\substack{I\subset \{1,\dots, c\}\\ \# I< \dim X}}\tilde{D}_I \).
\item \label{part2}Given  a smooth subvariety \(Z\subset X\) intersecting \(D\) transversally (so that \((Z,D|_Z)\) is a smooth pair) and assuming that \(c\leqslant \dim Z\), the resolution \(\widehat{Z}_1( D|_Z)\) is the strict transform in \(\widehat{X}_1( D)\) of \({Z}_1(D|_Z)\subset {X}_1(D).\)
\item \label{part3}Given a family of smooth pairs \((\X,\ds)\stackrel{\rho}{\to} S\) such that the number of components of \(\ds\) is less than the dimension of the fibers of \(\rho\), there exists a smooth variety \(\widehat{\P}(\Omega_{\X/S}(\log \ds))\) with a birational morphism \(\mu^{\rm rel}:\widehat{\P}(\Omega_{\X/S}(\log \ds))\to {\P}(\Omega_{\X/S}(\log \ds))\) 
such that for any \(s\in S\), denoting \(X_{s}=\rho^{-1}(s)\) and \(D_s=\ds|_{X_s}\), and viewing \(X_{s,1}(D_s)\) as a subvariety of \({\P}(\Omega_{\X/S}(\log \ds))\), one has 
\[(\mu^{\rm rel})^{-1}(X_{s,1}(D_s))\cong \widehat{X}_{s,1}(D_s).\]
 Moreover, there exists effective divisors  \(\mathscr{E}_1,\dots, \mathscr{E}_m\) on \(\widehat{\P}(\Omega_{\X/S}(\log \ds))\) such that for any \(s\in S\), the set of  irreducible exceptional divisor of the map \({\mu}:\widehat{X}_{s,1}(D_s)\to {X}_{s,1}(D_s)\)  is 
 \[\Big\{\mathscr{E}_1|_{\widehat{X}_{s,1}(D_s)},\dots, \mathscr{E}_m|_{\widehat{X}_{s,1}(D_s)}\Big\}.\]
\end{thmlist}
\end{proposition}
This allows us to make the following definition:
\begin{dfn} Let \((X,D)\) be a smooth pair. Suppose that \(X\) is projective. We say that the pair \((X,D)\) satisfies property \eqref{eq:star} if there exists a \(\mu\)-exceptional effective \(\qb\)-divisor \(F\in {\rm Div}_{\qb}(\widehat{X}_1(D))\) such that the \(\qb\)-line bundle 
\begin{equation}\label{eq:star}\tag{$\ast$}
\mu^*(\O_{X_1(D)}(1))\otimes \O_{\widehat{X}_1(D)}(-F)\ \ \text{is ample}.
\end{equation}
\end{dfn}

From \cref{prop:ResolutionAlgo} we deduce the following useful consequence. 
\begin{cor} 
\begin{thmlist}
\item Property \eqref{eq:star} is a Zariski open property. Namely,  given  a family of smooth pairs \((\X,\ds)\stackrel{\rho}{\to} S\), if there exists \(s_0\in S\) such that the pair \((X_{s_0},D_{s_0})\) satisfies  \eqref{eq:star},  then there exists a non-empty Zariski open subset \(U\subset S\) such that for any \(s\in U\) the pair \((X_s,D_s)\) satisfies \eqref{eq:star}.
\item \label{implies ample} Let \((X,D)\) be a smooth pair and et \(L\) be an ample line bundle on \(X\). If \((X,D)\) satisfies  \eqref{eq:star}, then the  logarithmic cotangent bundle \(\Omega_X(\log D)\) is almost ample.
\end{thmlist}
\end{cor}
\begin{proof} The first claim follows from \Cref{part3} by the openness property of ampleness. To prove the second claim, one can observe that if \(\mu^*(\O_{X_1(D)}(1))\otimes \O_{\widehat{X}_1(D)}(-F)\) is ample, then \(\mathbf{B}_+\big(\mu^*(\O_{X_1(D)}(1))\big)\subset \Supp(F)\) and therefore since \(\mu\) is birational and \(X_1(D)\) is smooth.
\[\mathbf{B}_+(\O_{X_1(D)}(1))\subset \mu\big(\Supp(F)\big)\subset \bigcup_{\substack{I\subset \{1,\dots, c\}\\ \# I< \dim X}}\tilde{D}_I,\]
in view of   \Cref{part1}.
\end{proof}
\subsection{Logarithmic connections}\label{meromorphic connection}
Let $L$ be a line bundle over a (not necessarily compact) complex manifold $X$.  Take a smooth hypersurface $D\in |L|$ (if such a hypersurface exists) and let $s_D\in H^0(X,L)$ be a section defining $D$. There exists a logarithmic  connection 
\begin{eqnarray}\label{meromorphic}
\nabla_{s_D}:L\rightarrow L\otimes \Omega_X(\log D),
\end{eqnarray}
defined locally by
\[\nabla_{s_D}s\stackrel{\rm loc}{:=}ds-s\frac{ds_D}{s_D}.\]

By this we mean that over an open subset \(U\) with a fixed trivialization of \(L\), if we let \(s_U,s_{D,U}\in \O(U)\) to be local representative for \(s\in \Gamma(U,L)\) and \(s_D\in \Gamma(X,L)\) under our choice of trivialization then 
\[\nabla_{s_D}s := ds_U-s_U\frac{d s_{D,U}}{s_{D,U}}.\]
One verifies without difficulty that this local definition defines a logarithmic connection on \(L\) with logarithmic poles along \(D\). 

Let us observe the following tautological relation 
\begin{equation}\label{eq:TautologicalVanishing}
\nabla_{s_D}(s_D)=0.
\end{equation}
Moreover, if \(Y\subset X\) is a smooth subvariety transverse to \(D\), and if one denotes \(s_{Y,D}:=s_D|_Y\), then one has the following commutative diagram
\[
\xymatrix{
L\ar[r]^{\!\!\!\!\!\!\!\!\!\!\!\!\!\!\! \!\!\!\!\!\!\!\!\!\!\nabla_{s_D}}\ar[d]& L\otimes \Omega_X(\log D)\ar[d]\\
L|_Y\ar[r]_{\!\!\!\!\!\!\!\!\!\!\!\!\!\!\!\!\!\!\!\!\nabla_{s_{Y,D}}} & L|_Y\otimes \Omega_Y(\log D|_Y),
}
\]
where the right vertical arrow is induced by the composition
\[ \Omega_X(\log D)\to \Omega_X(\log D)|_Y\to  \Omega_Y(\log D|_Y).\]

\section{Proof of the main results}\label{se:Main}
\subsection{The main technical result}
We will establish the following stronger version of our main result. 
\begin{thm}\label{maintechnical}
Let \(Y\) be a projective variety of dimension \(n\) and let \(L\) be a very ample line bundle on \(Y\). Take  \(\ep_1,\dots,\ep_n\in\nb^*\) and take  integers \(\delta_1,\dots,\delta_n\geqslant 4n-1\). Then, for any \(i\in \{1,\dots,n\}\) set \(b_i=\delta_i^{-1}\prod_{j=1}^n\delta_j\). 
For any 
\[r>\sum_{i=1}^nb_i(\ep_i+\delta_i)\] 
and for general hypersurfaces \(H_1\in |L^{\ep_1+(r+1)\delta_1}|,\dots,H_n\in |L^{\ep_n+(r+1)\delta_n}|\), writing \(H=\sum_{i=1}^nH_i\), the logarithmic cotangent bundle \(\Omega_Y(\log H)\) is almost ample. Moreover, the complement \(U=Y\setminus H\) is hyperbolically embedded in \(Y\).
\end{thm}
As a corollary of this result we obtain the following strengthening of \cref{cor:simple}.
\begin{cor}\label{main}
Let \(Y\) be a projective variety of dimension \(n\) and let \(L\) be a very ample line bundle on \(Y\). Take an integer  \(c\geqslant n\) and integers \(\delta_1,\dots,\delta_c\in \nb^*\) such that \(\delta_1,\dots,\delta_c\geqslant 4n-1\). Let \(\alpha\geqslant 3+2n(\max_{1\leqslant i\leqslant c}{\delta_i})^n\)  be a rational number such that \(\alpha \delta_i\in \nb\) for all \(i\in \{1,\dots, n\}\), and set  \((m_1,\dots,m_c)=\alpha\cdot (\delta_1,\dots, \delta_c)\in \nb^n\). For general hypersurfaces \(H_1\in|L^{m_1}|,\dots,H_c\in|L^{m_c}|\), writting \(D=\sum_{i=1}^cH_i\), the pair \((Y,D)\) has almost ample logarithmic cotangent bundle. 
\end{cor}
\begin{proof}[\cref{maintechnical} $\Rightarrow$ \cref{main} and  \cref{cor:simple}]
Let us first fix  
\[r_0:=1+2n\max_{1\leqslant i \leqslant c}(\delta_i)^n.\]
In view of  \cref{cor:MoreComponents} we are reduced to prove that under the hypothesis of \cref{main}, for every divisor \(D'\) consisting of \(n\) of the components of \(D\), the logarithmic cotangent bundle of \((Y,D')\) is almost ample. Without loss of generality, we may assume \(D'=\sum_{i=1}^nH_i\).

Our hypothesis on \(r_0\) guaranties that \(r_0>\sum_{i=1}^n2b_i\delta_i\), and therefore one can apply \cref{maintechnical} for any \(r\geqslant r_0\) and for any \((\ep_1,\dots,\ep_n)\) such that \(1\leqslant \ep_i\leqslant \delta_i\) for all \(i\in\{1,\dots, n\}\). 

For any rational number \(\alpha\geqslant r_0+2\), and any \((m_1,\dots,m_n)\in \nb^n\) of the form \(\alpha(\delta_1,\dots, \delta_n)\), one can set \(r:=\lceil\alpha\rceil-2\geqslant r_0\) and  for any \(i\in\{1,\dots, n\}\) one can set 
\(\ep_i=(\alpha-\lceil\alpha\rceil+1)\delta_i\in\{1,\dots, \delta_i\}\), so that \(m_i=\ep_i+(r+1)\delta_i\). Applying \cref{maintechnical} implies that the logarithmic cotangent bundle of \((Y,D')\) is almost ample, which concludes the proof of \Cref{main}.

To prove  \cref{cor:simple}, it suffices to take \(\delta_1,\dots,\delta_c=4n-1\).  The bound for \(\alpha\) is \(\alpha\geqslant 3+2n(4n-1)^n\). And we can therefore apply \Cref{main} to obtain the almost ampleness of the logarithmic cotangent bundle for degree  \[m=m_1=\dots= m_c\geqslant (4n-1)(3+2n(4n-1)^n).\] It then remains to observe that \((4n-1)(3+2n(4n-1)^n)\leqslant (4n)^{n+2}\).
\end{proof}

\subsection{Notation and conventions}
Let us summarize here the main notation and conventions we will use in \cref{se:Main}. 
We  fix an integer \(n\geqslant 2\) and take homogenous coordinates \([x_0,\dots, x_n]\) on \(\P^n\). Given any \(\delta\in \nb^*\), these homogenous coordinates induce an isomorphism
\[H^0(\P^n,\O_{\P^n}(\delta))=\cb[x_0,\dots, x_n]_{\delta},\]
where the right hand side denotes the set of homogenous polynomials of degree \(\delta\) in \(n+1\) variables.  The set of unitary monomials of degree \(\delta\) forms a basis of this space, and this basis  is naturally in bijection with the  set of multi-indices of weight \(\delta\)
 \[\ib(\delta):=\{I=(i_0,\dots, i_n)\in \nb^{n+1}\ | \ |I|=i_0+\cdots+i_n=\delta\}.\]
For any \(I=(i_0,\dots,i_n)\in \ib(\delta)\), we use the standard multi-index notation \(x^I:=x_0^{i_0}\cdots x_n^{i_n}\). We therefore have the following identification 
\begin{eqnarray*}
\cb^{\ib(\delta)}&\stackrel{\sim}{\to}& H^0(\P^n,\O_{\P^n}(\delta))\\
(a_I)_{I\in \ib(\delta)}&\mapsto& \sum_{|I|=\delta}a_Ix^I.
\end{eqnarray*}
This  induces an identification \(P(\cb^{\ib(\delta)})\cong |\O_{\P^n}(\delta)|\). We will implicitly use these identifications in what follows.

Given  a subset \(J\subset \{0,\dots,n\}\) of cardinality \(\# J=n-k\), we consider the \(k\)-dimensional projective subspace \(\P_J\subset \P^n\) defined by \[\P_J=\{[x]\in \P^n\ |\ x_j=0\ \forall j\in J\}.\]
The coordinates \([x_0,\dots, x_n]\) induce homogenous coordinates \([x_{\ell_0},\dots, x_{\ell_k}]\) on \(\P_J\), where \(\{\ell_0,\dots, \ell_k\}=\{0,\dots, n\}\setminus J\). Moreover, we have a restriction map
\[\res_J^{\delta}:H^0(\P^n,\O_{\P^n}(\delta))\to H^0(\P_J,\O_{\P_J}(\delta)),\]
which, with our choice of homogenous coordinates, is the map given by
\[\res^{\delta_i}_J\Big(\sum_{|I|=\delta_i}a_Ix^I\Big)=\sum_{\substack{|I|=\delta_i\\ \Supp(I)\cap J=\varnothing}}a_Ix^I.\]
where \(\Supp(I)=\{j\in \{0,\dots,n\}\ |\ i_j\neq 0\}\).
Writing
\[\ib_J(\delta):=\{I\in \nb^{n+1}\ | \ |I|=\delta\ \ \text{and}\ \ \Supp(I)\cap J=\varnothing\},\]
we can therefore identify the restriction map \(\res_J^{\delta}\) with the natural projection \(\cb^{\ib(\delta)}\to \cb^{\ib_J(\delta)}\).

\subsection{Setting} \label{construction}

Following the ideas of \cite{BD15}, we will construct $n$ families of divisors parametrized by certain Fermat-type equations. Let $L$ be a very ample line bundle over $Y$.  Fix $n+1$ sections in general position $\tau_0,\ldots,\tau_n\in H^0(Y,L)$. By ``general position" we mean that the divisor $\sum_{i=0}^n(\tau_i=0)$ is simple normal crossing. We fix  two $n$-tuples of positive integers $\bm{\varepsilon}=(\varepsilon_1,\ldots,\varepsilon_n)$ and $\bm{\delta}=(\delta_1,\ldots,\delta_n)$. For   any $i\in\{1,\ldots,n\}$ we define  
\[\A_i:=\bigoplus_{I\in \ib(\delta_i)}H^0(Y,L^{\varepsilon_i})\cong H^0(Y,L^{\ep_i})\otimes H^0(\P^N,\O_{\P^N}(\delta_i)).\]
Let us also fix a positive integer $r$ and consider for any \(\af_i=(a_{i,I})_{I\in \ib(\delta_i)}\) the hypersurface  $H_{\af_i}$ in $Y$   defined by the zero locus of the  section 
 \begin{eqnarray}\nonumber
 \sigma_i(\af_i):=\sum_{|I|=\delta_i}a_{i,I}\tau^{(r+1)I}\in H^0(Y,L^{m_i})
 \end{eqnarray}
 where $m_i=\varepsilon_i+(r+1)\delta_i$ and $\tau^{(r+1)I}:=(\tau_0^{i_0}\cdots\tau_n^{i_n})^{r+1}$ for \(I=(i_0,\dots,i_n)\). 
Define
 \[\A:=\A_1\times\cdots\times \A_n,\]
 and for any \(\af=(\af_1,\dots,\af_n)\in \A\), define \(H_{\af}:=\sum_{i=1}^cH_{\af_i}\).  Let us now observe that this way we obtain indeed smooth pairs. 
   \begin{lem}\label{lem:Bertini}
  There exists a  non-empty Zariski open subset  \(\A^{\sm} \subset \A \) such that for any \(\af\in \A^{\rm sm}\), the pair \((Y,H_{\af})\) is a smooth pair.
  \end{lem}
  \begin{proof} In \cite{BD15} (Lemma 2.1) was proven the following Bertini type result: Given any smooth subvariety \(W\subset Y\) and any \(i\in \{1,\dots,n\}\), then for a general \(\af_i\in \A_i\), the intersection \(W\cap H_{\af_i}\) is smooth. From there it suffices to make an induction on the number of components since this result ensures that at each step one can chose the hypersurface \(H_{\af_{i+1}}\) to be transverse to the configuration given by the first \(i\) constructed hypersurfaces \(H_{\af_1},\dots,H_{\af_i}\).
  \end{proof}
For any \(i\in \{1,\dots, n\}\), let us define $L_i:=L^{m_i}$ and let us denote by \(\lb_i\) the total space of \(L_i\). Moreover, let us denote by $\V$ the total space of the rank $n$ vector bundle $L_1 \oplus L_2\oplus\ldots\oplus L_n$ and let \(p_{\V}:\V\to Y\) be the canonical projection. For any \(i\in \{1,\dots, n\}\), denote by \(T_i\in H^0(\V,\p_{\V}^*L_i)\) the tautological section defined by 
 \[T_i(\ell_1,\dots,\ell_\c):=\ell_i\ \ \ \forall\ (\ell_1,\dots,\ell_\c)\in \V=\lb_1\times_Y\cdots\times_Y\lb_\c.\] 
 These induce a  simple normal crossing divisor \(\D=\sum_{i=1}^\c D_i\) on \(\V\) where for each \(i\in\{1,\dots,\c\}\)  we write \(D_i=(T_i=0)\).
 For any \(\af_i\in \A_i\), consider the section
 \[\Si_i(\af_i):=T_i-p_{\V}^*\sigma_i(\af_i)\in H^0(\V,p_{\V}^*L_i).\]
 Denote also by $E_{\af_i}=(\Si_i(\af_i)=0)\subset \V$   the hypersurface defined  by \(\Si_i(\af_i)\), and for  $\af=(\af_1,\dots,\af_\c)\in \A$ denote  \[Z_\af:=\bigcap_{i=1}^{\c}E_{\af_i}\subset \V.\]
If $H_{\af}$ is simple normal crossing, then so is $\sum_{i=1}^{\c}E_{\af_i}$, and therefore
   $Z_\af$ is a smooth complete intersection subvariety of $\V$. 
   For any \(i\in \{1,\dots, n\}\) set \(\D_{i,\af}:= D_i|_{Z_{\af}}\) and $\D_{\af}:=\sum_{i=1}^{\c}\D_{i,\af}$. By construction, the restriction of  $\p_{\V}$ to $Z_{\af}$
   $$\p_{\V}|_{Z_{\af}}:(Z_{\af}, D_{\af})\rightarrow (Y,H_{\af})$$
   is a  biholomorphism of  log manifolds.  
   
We will consider these varieties in families. Let us denote by \(\pr_1:\A\times \V\to \A\) and \(\pr_2:\A\times \V\to \V\) the two canonical projections, and consider for each \(i\in \{1,\dots, n\}\), the section \(\Si_i\in H^0(\A\times\V,\pr_2^*\p_{\V}^*L_i)\) defined by \(\Si_i(\af,v):=\Si_i(\af_i)(v)\) for any \(\af:=(\af_1,\dots,\af_c)\in \A\) and any \(v\in \V\).
   We then consider the family of complete intersections \(Z_{\af}\) as above, namely the  subvariety ${\zs}\subset \A\times \V$  defined by 
  \begin{eqnarray*}
  {\zs}:&=&\{(\af,v)\in \A\times \V\mid  \Si_1(\af_1,v) =\ldots= \Si_c(\af_c,v)=0\}.
  \end{eqnarray*}
On \(\zs\) we consider the divisor \(\ds:=\pr_2^*D\) which satisfies \(\ds|_{Z_{\af}}=D_{\af}\) for any \(\af\in \A^{\rm sm}\). We will denote by \[\zs_1^{\rm rel}(\ds):=\P(\Omega_{\zs/\A^{\rm sm}}(\log \ds)),\]
the projectivisation of the relative logarithmic cotangent bundle of \((\zs,\ds)\to \A^{\rm sm}\), the family \((\zs,\ds)\) restricted to \(\A^{\rm sm}\) and we will denote by 
\[\widehat{\zs}_1^{\rm rel}(\ds):=\widehat{\P}(\Omega_{\zs/\A^{\rm sm}}(\log \ds)),\]
the associated resolution constructed in \cref{prop:ResolutionAlgo}. This proposition moreover allows us to identify , for any \(\af\in \A^{\sm}\),  \(\widehat{Z}_{\af,1}(D_\af)\) with  the fiber in \(\widehat{\zs}_1^{\rm rel}(\ds)\) above \(\af\).


 \subsection{Modified logarithmic connections}\label{sse:ModifiedConnection}  In this section, let us fix \( i\in \{1,\dots, n\}\).  Recall that we defined \(\lb_i\) to be the total space of the line bundle \(L_i=L^{m_i}\). Let us denote by \(\p_i:\lb_i\to Y\) the canonical projection
and by \(T_{\lb_i}\in H^0(\lb_i,\p_i^*L_i)\) the tautological section defined by \(T_{\lb_i}(\ell_i):=\ell_i\) for any \(\ell_i\in \lb_i\). Note that one can identify \(Y\) with the zero locus of \(T_{\lb_i}\),  \(Y=(T_{\lb_i}=0)\). We will use this identification and study the log pair \((\lb_i,Y)\). For shortness, denote by  \(\pi_i:\lb_{i,1}(Y)\to \lb_i\)  the canonical projection. Observe that under the canonical projection \(q_i:\V\to \lb_i\), one has \(q_i^*T_{\lb_i}=T_i\) and \(q_i^*Y=D_i\). Therefore we get a morphism of log pairs \(q_i:(\V,D_i)\to (\lb_i,Y)\).
 Let us denote  by \[\nabla_i:\p_i^*L_i\to \p_i^*L_i\otimes \Omega_{\lb_i}(\log Y)\] the logarithmic connexion associated to the section \(T_{\lb_i}\) as defined in \eqref{meromorphic}. By \eqref{eq:TautologicalVanishing} one has \(\nabla_i(T_{\lb_i})=0\).
 \begin{lem}
 	For any $I\in \ib_i$, there exists a $\cb$-linear map
 \begin{eqnarray}\label{connection}
  	\nabla_{i,I}:H^0(Y,L^{\varepsilon_i})&\rightarrow& H^0\big(\mathbb{L}_i,\Omega_{\mathbb{L}_i}(\log Y)\otimes p^*_i L^{\varepsilon_i+\delta_i}\big)\\\nonumber
  	 a&\mapsto& \nabla_{i,I}(a)
 \end{eqnarray}
 	such that 
 	$$
 		p_i^*(\tau^{rI})\nabla_{i,I}(a)=\nabla_i (p_i^*a \cdot p_i^*\tau^{(r+1)I}).
 	$$
 \end{lem}
 \begin{proof}
One just needs to verify that \(\nabla_i (p_i^*a \cdot p_i^*\tau^{(r+1)I})\) is divisible by  \(p_i^*(\tau^{rI})\). This can be done locally: over an open subset \(U\) with a fixed choice of trivialization for \(L|_U\), denoting by \(a_U, {\tau}^I_U\in \O(U)\) the holomorphic function associated to \(a, \tau^I\) under this choice of trivialization, and denoting by \(t_i\in \O(p_i^{-1}(U))\) the holomorphic function associated to \(T_{\lb_i}\), one has
\[\nabla_i (p_i^*a \cdot p_i^*\tau^{(r+1)I})\stackrel{\rm loc}{=}p_i^*\tau_U^{rI} \big((r+1)p_i^*a_Ud(p_i^*\tau_U^I)+p_i^*\tau_U^I(d(p_i^*a_U)-p_i^*a_U\frac{dt_i}{t_i})\big)=p_i^*\tau_U^{rI} \cdot \nabla_{i,I}(a)_U.\]
 Since  \(\nabla_i (p_i^*a \cdot p_i^*\tau^{(r+1)I})\in  H^0\big(\mathbb{L}_i,\Omega_{\mathbb{L}_i}(\log Y)\otimes p^*_i L^{\varepsilon_i+(r+1)\delta_i}\big)\) and \(\tau^{rI}\in H^0(Y,L^{r\delta_i})\), we see that the elements \(\nabla_{i,I}(a)_U\) defined this way induces the desired global twisted logarithmic differential form \(\nabla_{i,I}(a)\).
  \end{proof}
 This allows us  to define the rational map
\begin{eqnarray}\label{rational map}
\varPhi_i:\A_i\times {\lb}_{i,1}(Y)&\dashrightarrow& |\O_{\P^n}(\delta_i)|\\ \nonumber
(\af_i,[\xi])&\mapsto&  \Big[\sum_{I\in\ib(\delta_i)}\nabla_{i,I}(a_{i,I})(\xi)x^I\Big].
\end{eqnarray}
This is well defined since it is independent of the choice of \(\xi\) representing \([\xi]\). We will now study the indeterminacy locus of this map.

 Take an open set $U\subset Y$ with local coordinates $(z_1,\ldots,z_n)$ and a fixed trivialization of  $L|_U“$. Then there are  induced coordinates $(t_i,z_1,\ldots,z_n)$ on the open set $\p_i^{-1}(U)$ in $\lbb_i$ with $Y\cap p^{-1}_i(U)=(t_i=0)$. For any \(\ell\in p^{-1}_i(U)\) and any non-zero $\xi \in T_{\lb_i}(-\log D_{\lb_i})_{\ell}$, consider the evaluation map
 \begin{eqnarray}
 	 \nabla^{\xi}_{i}:\A_i&\rightarrow& H^0(\P^n,\O_{\P^n}(\delta_i)).\nonumber\\
 	 \af_{i}&\mapsto& \sum_{|I|=\delta_i}\nabla^{\xi}_{i,I}(a_{i,I})x^I,	 \label{eq:Nabla_i}
 \end{eqnarray}
where \(\af_i=(a_{i,I})_{I\in \ib(\delta_i)}\) and where \(\nabla^{\xi}_{i,I}(a_{i,I})\) corresponds to the local representation of  \(\nabla_{i,I}(a_{i,I})(\xi)\) obtained from our choice of trivialization. Namely, for any \(a\in H^0(Y,L^{\ep_i}),\) if we denote by \(a_U\in \O(U)\) the corresponding holomorphic function under our choice of trivialization, then  
\begin{eqnarray}
\nabla_{i,I}^{\xi}(a)&:=& \big((r+1)p_i^*a_Ud(p_i^*\tau_U^I)+p_i^*\tau_U^I(d(p_i^*a_U)-p_i^*a_U\frac{dt_i}{t_i})\big)(\xi)\label{eq:NablaLocal}\\
&=& \big((r+1)a_Ud\tau_U^{I}+\tau_U^Ida_U\big)(dp_i(\xi))-p_i^*(\tau_U^Ia_U)\frac{dt_i}{t_i}(\xi)\nonumber.
\end{eqnarray}
Observe that \(\nabla_{i,I}^{\xi}:H^0(Y,L^{\ep_i})\to \cb\) is linear for each \(I\).

We are now going to estimate the rank of \(\nabla_i^{\xi}\). Let us first introduce a natural stratification on \(Y\). Given any \(J\subset \{0,\dots, n\}\), define \[Y_J:=\left\{y\in Y\ |\ \tau_j(y)=0\Leftrightarrow j\in J\right\}.\]
Observe that since \(\tau_0,\dots,\tau_n\) are supposed to be in general position, each \(Y_J\) is smooth of dimension  \(n-\# J\). Moreover, the family \((Y_J)_J\subset\{0,\dots, n\}\) is a stratification of \(Y\). 

 \begin{lem}\label{rank} Same notation as above. Take \(\ell\in \lb_i\) and a non-zero \(\xi\in T_{\lb_i}(-\log Y)_\ell\). Set \(y=p_i(\ell)\) and let \(J\subset \{0,\dots, n\}\) be such that  \(y\in Y_J\). Set \(k=n-\# J\). Then
 	$$
 	\rank\nabla_i^{\xi}\geqslant \rank \big(\res^{\delta_i}_J\circ \nabla_i^{\xi}\big)\geqslant   \binom{k+\delta_i}{k}.
 	$$
	
 \end{lem}
\begin{proof} The first inequality being obvious, we only prove the second one.
Observe that the linear map 
\begin{eqnarray}\label{restrict}
\res^{\delta_i}_J\circ \nabla_i^{\xi}=(\nabla_{i,I}^{\xi})_{I\in \ib_J(\delta_i)}
\end{eqnarray}
 is diagonal by blocs. 
Therefore we just have to study the rank of \(\nabla_{i,I}^{\xi}\) for each \(I\) such that \(|I|=\delta_i\) and \(\Supp(I)\cap J=\varnothing\). 
There are precisely \(\binom{k+\delta_i}{k}\) such multi-indices \(I\), and therefore it suffices to prove that for each of these, \(\rank(\nabla_{i,I}^{\xi})=1\) or equivalently, \(\nabla_{i,I}^{\xi}\neq 0\). 

Fix \(I\) such that \(\tau^I(y)\neq 0\). If \(dp_i(\xi)=0\) then there exists \(\lambda\in \cb^*\) such that $\xi=\lambda\cdot t_i\frac{\d }{\d t_i}$.  Since $L$ is very ample and \(\ep_i\geqslant 1\), there exists   $a\in H^0(Y,L^{\varepsilon_i})$ such that $a(y)\neq 0$. Thus (with  \(a_U\in \O(U)\) as above)
\begin{eqnarray*}
\nabla_{i,I}^{\xi}(a)&=&\big((r+1)a_Ud\tau_U^{I}+\tau_U^Ida_U\big)(dp_i(\xi))-(\tau_U^Ia_U)(y)\frac{dt_i}{t_i}(\xi)=-\lambda a_U(y) {\tau}^{I}(y)\neq 0.
\end{eqnarray*}
If $dp_i(\xi)\neq 0$, it also follows from the very ampleness of  $L$ that there exists   $a\in H^0(Y,L^{\varepsilon_i})$ such that  
$$a_U(y)=0\quad \mbox{and} \quad da_U\big(dp_i(\xi)\big) \neq 0.$$
Therefore \(\nabla_{i,I}^{\xi}(a)=\tau^I(y)da_U\big(dp_i(\xi)\big)\neq 0\). Hence the result.
%
%
\end{proof}
This lemma allows us to  control the indeterminacies of \(\varPhi_i\) in the following way.
\begin{proposition}\label{indeterminacy}
Assume  $\delta_i\geqslant 2n$. Then there exists a non-empty open set $\A_{i}^{\rm def}\subset \A_i$ such that the indeterminacy locus of $\varPhi_i$ does not intersect $\A_{ i}^{\rm def}\times\lb_{i,1}(Y)$.
\end{proposition}
\begin{proof} Let us define
\[B:=\left\{(\af_i,[\xi])\in \A_{i}\times \lb_{i,1}(Y)\ |\ \nabla_{i,I}(a_{i,I})(\xi)=0\ \forall I\right\}.\]
Certainly the indeterminacy locus of \(\Psi_i\) is contained in \(B\). We are now going to prove that \(B\) does not dominate \(\A_i\) under the projection \(\pr_1\) (here we denote by \(\pr_1\) and \(\pr_2\) the projections from \(\A_i\times\lb_{i,1}(Y)\) to each factor). To prove this we are going to prove that for each \(J\subset \{0,\dots, n\}\) the locus
\[B_J=B\cap (p_i\circ \pi_i\circ\pr_2)^{-1}(Y_J)\]
does not dominate \(\A_i\). This will be done by dimension count. 

Fix \(J\subset \{0,\dots, n\}\). Let us first suppose  that \(k:=n-\# J>0\). Fix \([\xi]\in (p\circ p_i)^{-1}(Y_J)\) and set \(B_{\xi}=B\cap\pr_2^{-1}([\xi])\). The projection \(\pr_1\) induces  an isomorphism
\[B_{\xi}\cong \pr_1(B_\xi)=\ker \nabla_i^{\xi}.\]
Thus by Lemma \ref{rank}, we obtain, in view of our hypothesis on \(\delta_i\):  
\[\dim B_{\xi}=\dim \A_i-\rank \nabla_i^{\xi}\leqslant \dim \A_i-\binom{k+\delta_i}{k}\leqslant \dim \A_i-(k+\delta_i).\]
Therefore 
\[\dim B_J\leqslant \dim (p_i\circ \pi_i\circ\pr_2)^{-1}(Y_J)+\dim \A_i-k-\delta_i=2n+1-k+\dim\A_i -k-\delta_i<\dim \A_i.\]
Hence \(B_J\) can not dominate \(\A_i\).

Let us now suppose \(n-\# J=0\).  Recall that \(Y_J\) is of dimension \(0\).
 Fix \(y\in Y_J\). We will prove that \(B\cap (p_i\circ\pi_i\circ \pr_2)^{-1}(y)\) does not dominate \(\A_i\). This will prove that \(B_J\) does not dominate \(\A_i\) since there are only finitely many points in \(Y\).
 Without loss of generality,  we may assume that  $J=\{1,\dots, n\}$. Take an open neighborhood $U$ of $y$ in $Y$ with the coordinates $(z_1,\ldots,z_n)$ centered at \(y\) such that   $(\tau_j=0)=(z_j=0)$  for all $j\in\{1,\ldots,n\}$,  and such that $(\tau_0=0)\cap U=\varnothing$.
  Then there exists natural coordinates $(t_i,z_1,\ldots,z_n)$ for $p_i^{-1}(U)\simeq \cb\times U$ such that $Y\cap p_i^{-1}(U)=(t_i=0)$. In this setting,  one has an isomorphism 
\begin{eqnarray*}
 \cb\times U  \times \pb^{n} &\xrightarrow{\simeq}& 	\lb_{i,1}(Y)\\
 (t_i,z,[\xi_0,\ldots,\xi_n])&\mapsto&  \left(t_i,z,\left[\xi_0t_i\frac{\d }{\d t_i}+\xi_1\frac{\d }{\d z_1}+\cdots+\xi_n\frac{\d }{\d z_n}\right]\right).
\end{eqnarray*}
Let us now observe that  for any \(\xi=\xi_0t_i\frac{\d }{\d t_i}+\xi_1\frac{\d }{\d z_1}+\cdots+\xi_n\frac{\d }{\d z_n}\),   \eqref{eq:NablaLocal} implies (let us drop here the local notation \(a_U,\tau^I_U\) for readability)
\[\nabla_{i,I}^{\xi}(a)=\big((r+1)ad\tau^I-\tau^Ida\big)(\xi')+\xi_0a\tau^I,\]
where \(\xi'=\xi_1\frac{\d }{\d z_1}+\cdots+\xi_n\frac{\d }{\d z_n}\). This expression is independent of the variable \(t_i\). Therefore, there exists \(B_y\subset \A_i\times \P^{n}\) such that if one denotes by \(\pr_{14}:\A_i\times \cb\times \{y\}\times \P^{n}\to \A_i\times \P^{n}\) the canonical projection, under our choice of coordinates,
\[B\cap (p_i\circ\pi_i\circ \pr_2)^{-1}(y)\cong \pr_{14}^{-1}(B_y).\]
We are therefore reduced to prove that \(B_y\) does not dominate \(\A_i\). Take \(I=(i_0,\dots, i_n)\in \ib(\delta_i)\) and \(\af_i\in \A_i\). If \(i_0\leqslant \delta_i-2\) then \(\nabla^\xi_{i,I}(a_{i,I})|_{y}=0\). If \(I=(\delta_i,0,\dots,0)\) then 
\[f^0_y(\af_i,\xi):=\nabla^\xi_{i,I}(a_{i,I})|_{y}=\big((r+1)a_{i,I}\delta_i\tau_0^{\delta_i-1}d\tau_0-\tau_0^{\delta_i}da_{i,I}\big)(\xi')+\xi_0a_{i,I}\tau_0^{\delta_0},\]
and if there exists \(j\in \{1,\dots,n\}\) such  \(I=(\delta_0-1,0,\dots, 1,\dots, 0)\) where the element ``\(1\)'' is at slot number \(j\), then 
\[f^j_y(\af_i,\xi):=\nabla^\xi_{i,I}(a_{i,I})|_{y}=(r+1)\frac{\d a_{i,I}}{\d z_j}\xi_j \tau_0^{\delta_i-1}.\] 
With this notation one has 
\[B_y=\{(\af_i,[\xi])\in \A_i\times \P^{n}\ |\ f^0_y(\af_i,\xi)=\cdots =f^n_y(\af_i,\xi)= 0\}\]
Hence, for any \([\xi]\in \P^{n}\), one has  \(B_y\cap (\A_i\times \{[\xi]\})=\ker\big(f_y(\cdot,\xi)\big)\), where \[f_y(\cdot,\xi)=\big(f^0_y(\cdot,\xi),\dots, f^n_y(\cdot,\xi)\big):\A_i\to \cb^{n+1}\] is the linear map induced by the different \(f^j_y(\cdot,\xi)\)'s.  Let us now stratify \(\P^{n}\) as follows: for any \(K\subset \{1,\dots,n\}\) define \[\Sigma(K):=\{[\xi]\in \P^{n}\ | \ \forall \ j\in \{1,\dots, n\}, \ \ \xi_j=0\ \Leftrightarrow\ j\in K\}.\]
One has \(\dim \Sigma(K)=n-\#K\). Moreover, for any \([\xi]\in \Sigma(K)\) one verifies easily that \(\rank f_y(\cdot,\xi)=1+n-\#K\), and therefore 
\[\dim B_y\cap (\A_i\times \Sigma(K))= \dim \A_i- (1+n-\#K)+\dim \Sigma(K)=\dim \A_i-1<\dim \A_i.\]
Therefore, for any \(K\subset \{1,\dots,n\}\), \( B_y\cap (\A_i\times \Sigma(K))\) does not dominate \(\A_i\), and since the family \(\Sigma(K)\) stratifies \(\P^{n}\), the locus \(B_y\) does not dominate \(\A_i\) either, which implies the result.
\end{proof}
\begin{rem}
In fact \cref{indeterminacy} still holds under the weaker assumption \(\delta_i\geqslant 3\), but the proof would require the use of a more intricate stratification, similar to the one used in \cite{BD15}. Since we will soon suppose that \(\delta_i\geqslant 4n-1\geqslant 2n\), we chose to present this non-optimal result in order to avoid those complications.
\end{rem}
Let us define \[\A^{\rm def}:=\big(\A_1^{\rm def}\times \cdots \times \A_\c^{\rm def}\big)\cap \A^{\rm sm}\subset \A.\]
Recall that \(q_i:(\V,D_i)\to (\lb_i,Y)\) is a morphism of smooth log pairs and therefore we obtain a morphism 
\[{}^tdq_i:q_i^*\Omega_{\lb_i}(\log Y)\to \Omega_{\V}(\log D_i).\]
Composing this map with the map \(\Omega_{\V}(\log D_i)\hookrightarrow \Omega_{\V}(\log D)\), twisting it by \(p_{\V}^*L^{\ep_i+\delta_i}\), and applying the global section functor, we obtain a \(\cb\)-linear map
\[H^0(\lb_i,\Omega_{\lb_i}(\log Y)\otimes p_{i}^*L^{\ep_i+\delta_i})\to H^0(\V,\Omega_{\V}(\log D)\otimes p_{\V}^*L^{\ep_i+\delta_i})\cong H^0({\V}_1(D), \O_{{\V}_1(D)}(1)\otimes \pi_{\V}^*p_{\V}^*L^{\ep_i+\delta_i})\]
%
%
For any  \(I\in \ib(\delta_i)\) and any  \(a\in H^0(Y,L^{\ep_i})\), we denote by 
\[\nabla'_{i,I}(a)\in H^0({\V}_1(D), \O_{{\V}_1(D)}(1)\otimes \pi_{\V}^*p_{\V}^*L^{\ep_i+\delta_i})\]
the image of \(\nabla_{i,I}(a)\) under the above composition. It follows from  \cref{resolve} that this global section vanishes along the ideal \(\js_i\) (with the notation of Definition \ref{obstruction} for the log pair \((\V,D)\)), namely 
\begin{equation}\label{eq:VanishingNabla}
\nabla'_{i,I}(a)\in H^0({\V}_1(D), \O_{{\V}_1(D)}(1)\otimes \pi_{\V}^*p_{\V}^*L^{\ep_i+\delta_i}\otimes \js_i).\end{equation}

%
Since the  resolution algorithm defined in \cref{prop:ResolutionAlgo} provides in particular a log resolution of $\js_i$, one has 
\begin{equation}\label{eq:DefF_i}
\mu^{*}\js_i=\O_{\widehat{\V}_1(D)}(-F_i)\end{equation}
for some effective divisor \(F_i\) on \(\widehat{\V}_1(D)\). Since \(\widehat{\V}_1(D)\) is smooth, it follows from \eqref{eq:VanishingNabla} that there exists 
\[\widehat{\nabla}_{i,I}(a)\in H^0\big(\widehat{\V}_1(D), \mu^*\big(\O_{\widehat{\V}_1(D)}(1)\otimes \pi_{\V}^*p_{\V}^*L^{\ep_i+\delta_i}\big)\otimes \O_{\widehat{\V}_1(D)}(-F_i)\big)\]
such that 
\[\widehat{\nabla}_{i,I}(a)=F_i\cdot\mu^*\nabla'_{i,I}(a).\]

Then, as in  \eqref{rational map}, we define a rational map 
\begin{eqnarray}\label{rational map2}
\widehat{\varPhi}_i:\A_i\times \widehat{\V}_1(D)&\dashrightarrow& |\O_{\P^n}(\delta_i)|\\\nonumber
(\af_i,\hat{w})&\mapsto&  \Big[\sum_{I\in \ib(\delta_i)}\widehat{\nabla}_{i,I}(a_{i,I})(\hat{w})x^I\Big].
\end{eqnarray}
By \cref{prop:ResolutionAlgo},  $\mu$ resolves the indeterminacies of the rational map $ \gamma_i: {\V}_1(D)   \dashrightarrow   {\V}_1(D_i)  $. Let us  denote by $\nu_i=\gamma_i\circ \mu$. By the definition of \(\widehat{\nabla}_{i,I}\) one can easily show that one has the following commutative diagram
\begin{equation}\label{Diagram}
\xymatrix{
\A_i\times \widehat{\V}_1(D) \ar^{\mathds{1}\times \mu}[d] \ar^{\mathds{1}\times \nu_i}[dr] \ar@{-->}@/^1pc/_{\widehat{\varPhi}_i}[drrr] & & & \\
\A_i\times {\V}_1(D)\ar@{-->}^{\mathds{1}\times \gamma_i}[r]  &\A_i\times {\V}_1(D_i)\ar@{-->}^{\mathds{1}\times [{}^tdq_i]}[r]  &\A_i\times {\lb_{i,1}}(Y) \ar@{-->}^{\varPhi_i}[r]  &|\O_{\P^n}(\delta_i)|,
}
\end{equation}
where \(\mathds{1}\) denotes the identity on \(\A_i\).
Let us now define \({\V}_1(D_i)^{\circ}\) to be the locus where \([{}^tdq_i]\) is regular (i.e. the complement of the indeterminacy locus of \([{}^tdq_i]\)), and let us define 
\[\widehat{\V}_1(D)^{\circ,i}:= \nu_i^{-1}({\V}_1(D_i)^{\circ}).\]
Then we have the following.

\begin{lem}\label{regular}
	For any $\af=(\af_1,\ldots,\af_\c) \in \A^{\sm}$, \(\widehat{Z}_{\af,1}(D_{\af})\subset \widehat{\V}_1(D)^{\circ,i}\). In particular, the restriction of \(\mathds{1}\times ([{}^tdq_i]\circ \nu_i)\) to \(\widehat{\mathscr{Z}}^{\rm rel}_1(\ds)\) is regular :
	\[\widehat{\mathscr{Z}}^{\rm rel}_1(\ds)\subset \A^{\sm}\times \widehat{\V}_1(D)^{\circ,i}{\longrightarrow} \A_i\times {\lb_{i,1}}(Y).\]
	
	\end{lem}
\begin{proof}
	By \cref{prop:ResolutionAlgo},  one has \( \nu_i(\widehat{Z}_{\af,1}(D_\af))= Z_{\af,1}(D_{i,\af}) \). It thus suffices to prove that \(Z_{\af,1}(D_{i,\af})\subset {\V}_1(D_i)^{\circ} \).
Let  \([\xi]\in \V_1(D_i)\) be  in the indeterminacy locus of \([{}^tdq_i]\). This means that  \(dq_i(\xi)=0\) in \(T\lb_i(-\log Y)\), from which we deduce that  \(\xi\neq 0 \) when seen as a vector in \(T\V \). Assume moreover that   \([\xi]\in Z_{\af,1}(D_{i,\af})\). Then \(d(p_\V|_{Z_{\af}})(\xi)\neq 0 \) in \(TY \) since  \(\p_{\V}|_{Z_{\af}}:Z_{\af}\rightarrow Y\) is a  biholomorphism.  But then from the commutative diagram 
\[\xymatrix{
	\V\ar^{q_i}[r]\ar_{p_{\V}}[d]& \lb_i\ar^{p_{i}}[dl]\\
	Y
}
\]
one deduces that  \(dq_i(\xi)\neq 0 \) in \( T\lb_i \), and  \emph{a fortiori} \(dq_i(\xi)\neq0\) in \(T\lb_i(-\log Y)\), which is a contradiction.
\end{proof}
Since we will be working simultaneously with all \(i\in\{1,\dots,\c\}\) we define  \[\widehat{\V}_1(D)^{\circ}:= \bigcap_{i=1}^\c\widehat{\V}_1(D)^{\circ,i}.\]
Define a rational map
\begin{eqnarray*}
\widehat{\varPhi}:\A \times \widehat{\V}_1(D)  &\dashrightarrow & |\O_{\P^n}(\delta_1)|\times \cdots\times |\O_{\P^n}(\delta_\c)|\\
(\af,\hat{w})&\mapsto& \big(\widehat{\varPhi}_1(\af_1,\hat{w}),\ldots,\widehat{\varPhi}_\c(\af_\c,\hat{w})\big)
\end{eqnarray*}
then by \cref{indeterminacy},  \eqref{Diagram} and \cref{regular},  its restriction to $ \A^{\rm def}\times\widehat{\V}_1(D)^{\circ} $ is regular. In particular, the restriction of $\widehat{\varPhi}$ to  $\widehat{\zs}_1^{\rm rel}(D)$ is also regular. 

\subsection{Universal complete intersections} Let us define  
\[\G:=|\O_{\P^n}(\delta_1)|\times \cdots\times |\O_{\P^n}(\delta_n)|\]
and let \(\ys\subset \G\times \P^n\) be the universal, smooth, complete intersection 
\[\ys:=\big\{(P_1,\dots,P_\c,[x])\in \G\times \P^n\ | \ P_1(x)=\cdots=P_\c(x)=0\big\}.\]
This object was of   critical importance in \cite{BD15,Bro17}, and  called the \emph{universal Grassmannian} in \cite{Den17}. 

Let us define
\begin{eqnarray*}
\widehat{\varPsi}:  \A^{\rm def} \times \widehat{\V}_1(D)^{\circ} &\to& \G\times \pb^n\\
(\af,\hat{w})&\mapsto& \big(\widehat{\varPhi}(\af,\hat{w}),[\tau_0(\hat{w})^r,\ldots,\tau_n(\hat{w})^r]\big).
\end{eqnarray*} 
where \(\tau_i(\hat{w}):= \tau_i\big(p_{\V}\circ \pi_{\V}\circ\mu(\hat{w})\big)\) for each \(i\in \{0,\dots, n\}\).
Let us also define \(\widehat{\zs}^{\rm rel}_1:=\widehat{\zs}^{\rm rel}_1(\ds)\cap \big(\A^{\rm def}\times \widehat{\V}_1(D)\big)\).
\begin{proposition}\label{prop:Restriction}When restricted to \(\widehat{\zs}^{\rm rel}_1\subset \A^{\rm def}\times \widehat{\V}_1(D)^{\circ}\), the morphism \(\widehat{\varPsi}\) factors through \(\ys\):
\[\widehat{\varPsi}|_{\widehat{\zs}^{\rm rel}_1}:\widehat{\zs}^{\rm rel}_1\to \ys\subset \G\times \P^n.\] 
\end{proposition}
\begin{proof}
We will prove that for any \(\af\in \A^{\rm def}\) one has \(\widehat{\varPsi}(\widehat{Z}_{\af,1}(D_{\af}))\subset \ys\). Recall the notation of Section \ref{construction} and Section \ref{sse:ModifiedConnection}. Fix \(\af=(\af_1,\dots, \af_n)\in \A^{\rm def}\). 
For any \(i\in \{1,\dots, n\}\) let us define 
\[E^i:=(T_{\lb_i}-p_i^*\sigma_i(\af_i)=0)\subset \lb_i.\]
Observe that by definition of \(Z_{\af}\), one has \([{}^tdq_i]\circ\nu_i(\widehat{Z}_{\af,1}(D_{\af}))\subset [{}^tdq_i]\big({Z}_{\af,1}(D_{i,\af})\big)\subset E_1^i(Y|_{E^i})\). Consider the map \(\varPsi^{i}:\lb_{i,1}(Y)\to |\O_{\P^n}(\delta_i)|\times \P^n\) defined by 
\[\varPsi^{i}([\xi]):=(\varPhi_i(\af_i,[\xi]),[\tau_0(y)^r,\dots,\tau_n(y)^r])\]
where \(y=p_i\circ \pi_{i}([\xi])\). From Diagram \eqref{Diagram} we are reduced to prove that for any \(i\in \{1,\dots, n\}\)  one has \(\varPsi^{i}\big(E_1^i(Y|_{E^i})\big)\subset\ys^i\) where 
\[\ys^i:=\{(P_i,[x])\in |\O_{\P^n}(\delta_i)|\times \P^n\ |\ P(x)=0\}.\] To see this, we apply the tautological relation \eqref{eq:TautologicalVanishing} to the connection \(\nabla_i\) associated to \(T_{\lb_i}\) and obtain
\[\nabla_i(T_{\lb_i})=0\in H^0(\lb_i,\Omega_{\lb_i}(\log Y)).\]
When restricted to  \(E^i\), this implies that 
\[0=\nabla_i(T_{\lb_i})=-\nabla_i(p_i^*\sigma_i(\af_i))=-\sum_{I\in \ib(\delta_i)} \nabla_{i,I}(a_{i,I})p_i^*(\tau^{rI})\]
when viewed as element in \(H^0\big(E^i,\Omega_{E^i}(\log Y|_{E^i})\otimes L^{\varepsilon_i+\delta_i}\big)\). From this it follows that for any \([\xi]\in E_1^i(Y|_{E^i})\) one has \(\varPsi^i([\xi])\in \ys^i\), hence the result.
\end{proof}
We will need to consider a slightly modified version of \(\ys\) on each strata \(Y_J\). For any \(J\subset \{0,\dots,n\}\) we define 
\[\ys_J:=\ys\cap (\G\times \P_J)\subset \G\times \P^n.\]
Define also \(\widehat{\V}_{1,J}(D)^{\circ}:=\widehat{\V}_1(D)^{\circ}\cap (p_{\V}\circ\pi_{\V}\circ \mu)^{-1}(Y_J)\) and \(\widehat{\zs}_{1,J}^{\rm rel}:=\widehat{\zs}_{1}^{\rm rel}\cap\widehat{\V}_{1,J}(D)^{\circ}\). Observe that by the very definition of \(Y_J\) and \(\P_J\), we have that for any \(y\in Y_J\), \([\tau_0(y)^r,\dots, \tau_n(y)^r]\in \P_J\). Therefore \(\widehat{\varPsi}(\A^{\rm def}\times \widehat{\V}_{1,J}(D)^{\circ})\subset \G\times \P_J\) and therefore, combining this with Proposition \ref{prop:Restriction} we deduce that \(\widehat{\varPsi}\) factors through \(\ys_J\) when restricted to \(\widehat{\zs}_{1,J}^{\rm rel}\):
\begin{equation}\label{eq:restriction}
\widehat{\varPsi}|_{\widehat{\zs}^{\rm rel}_{1,J}}:\widehat{\zs}^{\rm rel}_{1,J}\to \ys_J\subset \G\times \P_J.\end{equation}
Observe that the morphism \(\rho_J:\ys_J\to \G\) induced by the projection  \(\G\times \P_J\to \G\) is generically finite. Let us define  \[\G_J^{\infty}:=\big\{\Delta\in \G \ |\ \dim(\rho_J^{-1}(\Delta))>0\big\}\subset \G,\]
which is a closed subset of \(\G\).

Now we are ready to prove the following lemma:
\begin{lem}\label{exceptional locus}
	   Assume that   $\delta_i\geqslant 4n-1$ for any $i\in\{1,\dots,\c\}$. Then there exists a non-empty Zariski open subset $\A^{\rm nef}\subset \A^{\rm def}$ such that for any $J\subset \{0,\ldots,n\}$, 
	$$
	{\widehat{\varPhi}}^{-1}(\G_J^\infty)\cap (\A^{\rm nef}\times \widehat{\V}_{1,J}(D)^{\circ} )=\varnothing.
	$$
\end{lem}
\begin{proof} Take \(J\subset \{0,\dots, n\}\). 
	Fix   $\hat{w}\in \widehat{\V}_{1,J}(D)^{\circ}$ and set \(y=(p_{\V}\circ \pi_{\V}\circ \mu)(\hat{w})\in Y_J\). For any \(i\in \{1,\dots,\c\}\), define  \(\ell_i=(q_i\circ\pi_{\V,D_i}\circ \nu_i)(\hat{w})\) and let $\xi_i\in T_{\lbb_i}(-\log Y)_{\ell_i}$ such that \([\xi_i]=[{}^tdq_i]\circ\nu_i(\hat{w})\). 	Then \(p_i(\ell_i)=y \). The following commutative diagram summarizes the maps between the different spaces under consideration.
	\begin{equation*}
	\xymatrix{
		  \widehat{\V}_1(D)^{\circ} \ar^{  \mu}[d] \ar^{  \nu_i}[dr]  &   \\
	 {\V}_1(D)\ar@{-->}^-{\gamma_i}[r] \ar[dr]_-{\pi_\V} &  {\V}_1(D_i)^{\circ}  \ar[d]^-{\pi_{\V,D_i}}\ar@{->}^-{[{}^tdq_i]}[r] &  {\lb_{i,1}}(Y)\ar[d]^-{\pi_i} &\\
	& \V  \ar[r]^{q_i}\ar[dr]_-{p_\V} & \lb_i\ar[d]^-{p_i} &\\
	&& Y
	}
	\end{equation*}
	Fix a neighborhood \(U\subset Y\) of \(y\) with a fixed choice of trivialization of \(L|_U\) and consider the map (recall \eqref{restrict})
\begin{eqnarray*}
\res^{\delta_i}_J\circ \nabla_i^{\xi_i}:\A_i&\to& H^0\big(\P_J,\O_{\P_J}(\delta_i)\big)\\
\af_i&\rightarrow&\sum_{I\in \ib_J(\delta_i)}\nabla_{i,I}^{\xi_i}(a_{i,I})x^I.
\end{eqnarray*}
	Define \(\widetilde{\G}_J:=\prod_{i=1}^\c H^0\big(\P_J,\O_{\P_J}(\delta_i)\big)\) and define  
	\[\nabla_J^{\hat{w}}:=\big(\res^{\delta_1}_J\circ \nabla_1^{\xi_1},\dots, \res^{\delta_n}_J\circ \nabla_n^{\xi_n}):\A\to \widetilde{\G}_J.\]
	Consider moreover the complete intersection
	\[\widetilde{\ys}_J:=\big\{(P_1,\dots,P_\c,[x])\in \widetilde{\G}_J\times \P_J\ | \ P_1(x)=\cdots =P_\c(x)=0\big\},\]
and let us denote by \(\tilde{\rho}_J:\widetilde{\ys}_J\to \widetilde{\G}_J\) the morphism induced by the projection on the first factor.
Observe that \(\tilde{\rho}_J\) is generically finite and let us define
\[\widetilde{\G}_J^{\infty}:=\big\{\Delta\in \widetilde{\G}_J\ |\ \dim(\tilde{\rho}_J^{-1}(\Delta))>0\big\}.\]
This is a closed algebraic subset of \(\widetilde{\G}_J\). 
Let us now denote by \(\hat{\pr}_1:\A^{\rm def}\times \widehat{\V}_{1,J}(D)^{\circ}\to \A^{\rm def}\) and \(\hat{\pr}_2:\A^{\rm def}\times \widehat{\V}_{1,J}(D)^{\circ}\to \widehat{\V}_{1,J}(D)^{\circ}\) the projections on the first and second factor. Then \(\hat{\pr}_1\) induces an isomorphism 
\[\widehat{\varPhi}^{-1}(\G_J^{\infty})\cap \hat{\pr}_2^{-1}(\{\hat{w}\})\stackrel{\sim}{\to} \hat{\pr}_1\big(\widehat{\varPhi}^{-1}(\G_J^{\infty})\cap \hat{\pr}_2^{-1}(\{\hat{w}\})\big).\]
On the other hand, one has
\[ \hat{\pr}_1\big(\widehat{\varPhi}^{-1}(\G_J^{\infty})\cap \hat{\pr}_2^{-1}(\{\hat{w}\})\big)=(\nabla_J^{\hat{w}})^{-1}(\widetilde{\G}_J^{\infty})\cap \A^{\rm def}.\]
Indeed, for any \(\af\in \A^{\rm def}\), 
\[\widehat{\varPhi}(\af,\hat{w})=\big(\varPhi_1(\af_1,[\xi_1]),\dots, \varPhi_n(\af_n,[\xi_n])\big)=\big([\nabla_1^{\xi_1}(\af_1)],\dots, [\nabla_n^{\xi_n}(\af_n)]\big).\]
Therefore one has \(\rho_J^{-1}(\widehat{\varPhi}(\af,\hat{w}))=\tilde{\rho}_J^{-1}\big(\nabla_J^{\hat{w}}(\af)\big)\) and thus
\[\af\in \hat{\pr}_1\big(\widehat{\varPhi}^{-1}(\G^{\infty})\cap \hat{\pr}_2^{-1}(\{\hat{w}\})\big)\Leftrightarrow \Big(\af\in \A^{\rm def} \ \  \text{and}\ \ \dim\big(\rho_J^{-1}(\widehat{\varPhi}(\af,\hat{w}))\big)>0\Big)\Leftrightarrow \af\in (\nabla_J^{\hat{w}})^{-1}(\widetilde{\G}^{\infty}_J)\cap \A^{\rm def}.\] 
It follows from  a result of  Benoist, \cite{Ben11} Lemma 2.3 (see \cite{BD15} Section 3 for more details), that
\[\codim_{\widetilde{\G}_J}(\widetilde{\G}_J^{\infty})\geqslant \min_{1\leqslant i\leqslant c}\delta_i+1\geqslant 4n,\]
where the second inequality follows from our hypothesis on the \(\delta_i\)'s. Since \(p_i(\ell_i)=y\in Y_J \), by \cref{rank} this implies that for each \(i\in \{1,\dots, n\}\) the linear map \(\res^{\delta_i}_J\circ\nabla_{i}^{\xi_i}\) is surjective, and we obtain that 
\begin{eqnarray*}\dim\big((\nabla_J^{\hat{w}})^{-1}(\widetilde{\G}_J^{\infty})\big)&=&\dim \A+\dim\widetilde{\G}_J^{\infty}-\dim \widetilde{\G}_J\\
&=&\dim\A-\codim_{\widetilde{\G}_J}(\widetilde{\G}_J^{\infty})\leqslant \dim \A-4n.\end{eqnarray*}
Therefore we have 
\[\dim \big( \widehat{\varPhi}^{-1}(\G_J^{\infty})\cap \hat{\pr}_2^{-1}(\{\hat{w}\})\big)\leqslant \dim \A-4n.\]
Since this holds for any \(\hat{w}\in \widehat{\V}_{1,J}(D)^{\circ}\), and since \(\dim(\widehat{\V}_{1,J}(D)^{\circ})\leqslant \dim\big(\widehat{\V}_{1}(D)\big)=4n-1<4n\)
we obtain 
\[\dim\big(\widehat{\varPhi}^{-1}(\G_J^{\infty})\cap \hat{\pr}_2^{-1}(\widehat{\V}_{1,J}(D)^{\circ})\big)\leqslant \dim \A-4n+\dim \widehat{\V}_{1,J}(D)^{\circ}<\dim \A.\]
It follows that \(\widehat{\varPhi}^{-1}(\G_J^{\infty})\cap \hat{\pr}_2^{-1}(\widehat{\V}_{1,J}(D)^{\circ})\) does not dominate \(\A\), and thus there exists an non-empty Zariski open subset \(\A_J\subset \A^{\rm def}\) such that 
\[\widehat{\varPhi}^{-1}(\G_J^{\infty})\cap\big(\A_J\times \widehat{\V}_{1,J}(D)^{\circ}\big)=\varnothing.\]
It then suffices to take \(\A^{\rm nef}=\bigcap_{J\subset \{0,\dots,n\}}\A_J\).

\end{proof}
\subsection{Proof of the main results}\label{proof of main} For \(\G\) as above, and for \(a_1,\dots,a_n\in \mathbb{Z}\) we write
\[\O_{\G}(a_1,\dots, a_\c):=\O_{|\O_{\P^n}(\delta_1)|}(a_1)\boxtimes \cdots \boxtimes \O_{|\O_{\P^n}(\delta_\c)|}(a_\c).\]
From Nakamaye's Theorem on the augmented base locus, we know that for any \(J\subset \{0,\dots,n\}\) and for any \(a_1,\dots,a_\c>0\) and has
\[\mathbf{B}_+(\rho_J^*\O_{\G}(a_1,\dots, a_\c))={\rm Exc}(\rho_J)\subset \rho_J^{-1}(\G_J^{\infty}),\]
where \({\rm Exc}(\rho_J):=\big\{y\in \ys\ | \ \dim_y(\rho_J^{-1}(\rho_J(y)))>0\big\}\)  is certainly contained in \(\rho_J^{-1}(\G_J^{\infty})\). Recall that \(\rho_J:\ys\to \G\) is the map induced by the projection on the first factor \(q_1:\G\times \P^n\to \G\). Let us also denote by \(q_2:\G\times \P^n\to \P^n\) the projection on the second factor.

In \cite{Den17} the second named author established the following ``effective \emph{Nakamaye theorem}'' (which we present here in a weaker version):
\begin{thm}\label{nakmaye type}
	For any $\c$  positive integers $\delta_1,\ldots,\delta_\c\in \mathbb{N}^{*}$, if $a_i\geqslant  b_i:=\delta_i^{-1}\prod_{j=1}^{\c}\delta_j$ for each $i$, then for any $J\subset\{0,\ldots,n\}$, the base locus of the  line bundle $q_1^*\O_{\G}(a_1,\ldots,a_\c)\otimes q_2^*\oc_{\pb^{n}}(-1)|_{\ys_J}$ satisfies
\begin{eqnarray}\label{nakmaye}
	{\rm Bs}\Big(q_1^*\O_{\G}(a_1,\ldots,a_\c)\otimes q_2^*\oc_{\pb^{n}}(-1)|_{\ys_J}\Big)\subset \rho_J^{-1}(\G_{J}^{\infty}).
\end{eqnarray}
\end{thm}
 We are now going to prove the following result.
\begin{thm}\label{thm:maintechnical} Same notation as above. Suppose  \(\ep_1,\dots, \ep_\c\geqslant 1\) and \(\delta_1,\dots,\delta_\c\geqslant 4n-1\). Write \(b_i={\delta_i^{-1}}\prod_{j=1}^\c\delta_j\) for all \(i\in \{1,\dots, \c\}\).
If \[r>\sum_{i=1}^\c b_i(\ep_i+\delta_i),\]
then, for any \(\af\in \A^{\rm nef}\) the pair \((Z_{\af},D_{\af})\) satisfies Property \eqref{eq:star}.  
\end{thm}
\begin{proof} 
Let \(\af\in \A^{\rm nef}\). Let us first prove that 
\begin{equation}\label{eq:Goal1}
\widehat{\varPsi}^*\big(q_1^*\O_{\G}(b_1,\dots,b_\c)\otimes q_2^*\O_{\P^n}(-1)|_{\ys_J}\big)|_{\widehat{Z}_{\af,1}(D_{\af})} \ \ \text{is nef}.
\end{equation}
In order to prove that a line bundle is nef, it suffices to show that for any irreducible curve, its intersection with the given line bundle is non-negative. Let $C\subset \widehat{Z}_{\af,1}(D_{\af})$ be an irreducible curve. There exists a unique $J\subset \{0,\ldots,n\}$ such that $C^\circ:=C\cap \widehat{\zs}_{1,J}^{\rm rel}$ is a non-empty Zariski open subset of $C$. It follows from \eqref{eq:restriction} that $\widehat{\varPsi}(C^\circ)\subset \ys_J$, and since $\ys_J$ is closed in $\ys$, it follows that $\widehat{\varPsi}(C)\subset \ys_J$. 
By Lemma \ref{exceptional locus}, we have $\widehat{\varPsi}(C)\not\subset \rho_J^{-1}(\G_J^\infty)$ in view of our assumption of the \(\delta_i\)'s. Then,  from \eqref{nakmaye}  one obtains that 
	$$\widehat{\varPsi}(C)\not\subset {\rm Bs}\big(q_1^*\O_{\G}(b_1,\dots,b_\c)\otimes q_2^*\O_{\P^n}(-1)|_{\ys_J}\big).$$
Thus one has
$$
\widehat{\varPsi}^*\big(q_1^*\O_{\G}(b_1,\dots,b_\c)\otimes q_2^*\O_{\P^n}(-1)|_{\ys_J}\big)\cdot C\geqslant \big(q_1^*\O_{\G}(b_1,\dots,b_\c)\otimes q_2^*\O_{\P^n}(-1)|_{\ys_J}\big)\cdot \widehat{\varPsi}(C)\geqslant 0.
$$
Since $C$ is arbitrary, \eqref{eq:Goal1} is thus proved. 

Form the definition of the map \(\widehat{\varPsi}\) and the definition of the exceptional divisors \(F_1,\dots, F_n\), see \eqref{eq:DefF_i}, we obtain that, with the notation \(b:=\sum_{i=1}^\c b_i\),  \(b':=\sum_{i=1}^\c b_i(\ep_i+\delta_i)\) and \(F'=\sum_{i=1}^\c b_iF_i\)
\[\widehat{\varPsi}^*\big(q_1^*\O_{\G}(b_1,\dots,b_\c)\otimes q_2^*\O_{\P^n}(-1)|_{\ys_J}\big)|_{\widehat{Z}_{\af,1}(D_{\af})}=\big(\mu^*(\O_{\V_1(D)}(b)\otimes \pi_{\V}^*p_{\V}^*L^{b'-r})\otimes \O_{\widehat{\V}_1(D)}(-F')\big)\big|_{\widehat{Z}_{\af,1}(D_{\af})}.\]
Taking the \(\frac{1}{b}\)-th power of this last line bundle, and writing   \(\theta=\frac{1}{b}(r-b')\), we obtain that the \(\qb\)-line bundle 
\[\mu^*\big(\O_{\V_1(D)}(1)\otimes \pi_{\V}^*p_{\V}^*L^{-\theta}\big)\otimes \O_{\widehat{\V}_1(D)}\Big(-\frac{1}{b}F'\Big)\Big|_{\widehat{Z}_{\af,1}(D_{\af})}\]
is nef.  Our hypothesis on \(r\) precisely implies that \(\theta>0\). On the other hand, since \(\p_{\V}|_{Z_{\af}}:Z_{\af}\rightarrow Y\) is a  biholomorphism,  there exists \(\alpha \in \nb^*\)  and another \(\mu\)-exceptional effective \(\qb\)-divisor \(F''\) such that the \(\qb\)-line bundle \(\mu^*(\O_{\widehat{\V}_1(D)}(1)\otimes \pi_{\V}^*p_{\V}^*L^\alpha)\otimes \O_{\widehat{\V}_1(D)}(-F'')|_{\widehat{Z}_{\af,1}(D_\af)} \) is ample. Therefore  the line bundle 
\[\mu^*\Big(\O_{\V_1(D)}\Big(1+\frac{\theta}{\alpha}\Big)\Big)\otimes \O_{\widehat{\V}_1(D)}\Big(-\frac{1}{b}F'-\frac{\theta}{\alpha} F''\Big)\Big|_{\widehat{Z}_{\af,1}(D_{\af})}\]
is ample as sum of an ample and a nef line bundle. Let us denote by \(F_{\af}=\Big(\frac{\alpha}{b(\theta+\alpha)}F'+\frac{\theta}{\theta+\alpha} F''\Big)\big|_{\widehat{Z}_{\af,1}(D_{\af})}\), \(L_\af:= p_\V^*L|_{Z_\af} \), \(\pi_\af:Z_{\af,1}(D_\af)\rightarrow Z_\af \) the natural projection map, and \(\mu_\af:\widehat{Z}_{\af,1}(D_\af)\rightarrow  {Z}_{\af,1}(D_\af) \) the restriction of \(\mu\). By \Cref{part2}, \(\mu_\af\) is the minimal resolution of \( {Z}_{\af,1}(D_\af)\) and one has
\[
\mu^*\big(\O_{\V_1(D)}(1)\big)\otimes \O_{\widehat{\V}_1(D)}\Big(-\frac{\alpha}{b(\theta+\alpha)}F'-\frac{\theta}{\theta+\alpha} F''\Big)\Big|_{\widehat{Z}_{\af,1}(D_{\af})}=\mu_\af^*\big(\O_{Z_{\af,1}(D_\af)}(1)\big)\otimes \O_{\widehat{Z}_{\af,1}(D_\af)} (-F_{\af}),
\]
such that \(F_{\af}\) is \(\mu_\af\)-exceptional.  Whence the result.
\end{proof}
We can now prove \cref{maintechnical}.
\begin{proof}[Proof of \cref{maintechnical}]
Take \(\af \in \A^{\rm nef}\). Since \((Z_{\af},D_{\af})\) is biholomorphic to the pair \((Y,H_{\af})\) and that  \((Z_{\af},D_{\af})\) satisfies property \eqref{eq:star}, it follows that \((Y,H_{\af})\) satisfies property \eqref{eq:star}. Since Property \eqref{eq:star} is Zariski open, it follows that for general \(H_1\in |L^{m_1}|,\dots, H_n\in |L^{m_n}|\), writing \(H=\sum_{i=1}^nH_i\) the pair \((Y,H)\) satisfies \eqref{eq:star}, and therefore, by Corollary \ref{implies ample}, \((Y,H)\) has almost ample logarithmic cotangent bundle. 

Let us now see that this implies that \(Y\setminus H\) is hyperbolically embedded. First observe  that the complement \(Y\setminus H\) is Brody hyperbolic, which means that \(Y\setminus H\) doesn't contain any entire curve. This follows from a well known result in the theory of entire curves (see for instance \cite{Dem97,Kob98,McQ98,NW14}), which states that for any non-constant holomorphic morphism \(f:\cb\to Y\setminus H\), one has \(f(\cb)\subset \pi_Y(\mathbf{B}_+(\O_{Y_1(H)}(1)))\). Since in our case \( \pi_Y(\mathbf{B}_+(\O_{Y_1(H)}(1))= D\), such entire curves cannot exist. Moreover, for any \(I=\{i_1,\dots, i_r\}\subset\{1,\dots, n\}\) we can apply our result to the pair \((D_I,D(I^{\complement})|_{D_I})\) where we recall that \(D_I=D_{i_1}\cap\cdots \cap D_{i_r}\) and \(D(I^{\complement})=\sum_{i\in I^{\complement}}D_i\). Therefore the previous argument implies that for any such \(I\), the variety \(D_I\setminus D(I^{\complement})\) is Brody hyperbolic. A result of Green \cite{Gre77} now insures that this implies that \(Y\setminus H\) is hyperbolically embedded.
\end{proof}

\section{ Optimality on the number of components}\label{apprendix a}
In this section we mention two remarks concerning the number of components \(c\) of the divisor \(D\) in the case \(Y=\P^n\). 
First we show that if the ambient variety is \(\P^n\) then  our main result is optimal on the number of components. And then we recall that the logarithmic irregularity of the pair \((\P^n,D)\) is \(c-1\).

The following vanishing result,  is a  generalization to the logarithmic setting of a vanishing result for symmetric differential forms which was first established by Sakai \cite{Sak79} (see also \cite{BR90,Sch92}). We provide here a logarithmic adaptation of the arguments of \cite{Sak79} and  \cite{Bro11}.  The one component case \(c=1\) is a special case of a result proven in \cite{Div09} and the case \(n=2\) was already established in \cite{ElG03}.
\begin{proposition}\label{optimal c}
	Let $D:=\sum_{i=1}^{c}D_i$ be a simple normal crossing divisor in $\pb^n$. If \(c<n\) then for any \(m\geqslant 1\)
	\begin{eqnarray}\label{vanish}
	H^0\big(\pb^n,S^m\Omega_{\mathbb{P}^n}(\log D)\otimes \oc_{\pb^n}(-1)\big)=0
	\end{eqnarray}
In particular, $\Omega_{\mathbb{P}^n}(\log D)$ can not be almost ample.
\end{proposition}
\begin{proof}
For every \( i\in \{1,\dots, c\}\), set \(d_i=\deg D_i\)  and let us denote by $s_i\in H^0\big(\mathbb{P}^n,\oc_{\pb^n}(d_i)\big)$ the homogeneous polynomials defining the hypersurfaces $D_i$. One can view  $\pb^n\subset \pb^{n+c}$ as the subspace in $\mathbb{P}^{n+c}$ defined by $z_{n+1}=\ldots=z_{n+c}=0$. Set $D'=\sum_{i=1}^{c}D'_{i}$ with $D'_{i}:=\{[z]\in \pb^{n+c}\ |\ z_{n+i}=0\}$  for any $i\in\{ 1,\ldots, c\}$. Let $H_1,\ldots,H_c$ be $c$ hypersurfaces in $\mathbb{P}^{n+c}$ defined by $\{z_{n+i}^{d_i}-s_i(z_1,\ldots,z_n)=0\}_{i\in\{1,\ldots,c\}}$. Since \(D\) is simple normal crossing, the sum of these hypersurfaces is a simple normal crossing divisor and  therefore there intersection $X=H_1\cap\cdots\cap H_c$ is a smooth complete intersection. Let us define  \(E:=D'|_X\). Then $(X,E)$ is a sub-log pair of $(\pb^{n+c},D')$, and the natural projection $p:\pb^{n+c}\dashrightarrow \pb^n$ induces a  cover $\pi:X\rightarrow \mathbb{P}^n$ ramified along  $E$. One has a natural inclusion induced by \({}^td\pi|_X\),
		$$
		  H^0\big(\pb^n, S^m\Omega_{\pb^n}(\log D)\otimes \oc_{\pb^n}(-1)\big)\hookrightarrow H^0\big(X,S^m\Omega_X(\log E)\otimes \oc_{X}(-1)\big).
	$$
It then suffices to prove the vanishing of the right hand side.
	Set $N:=n+c$. Observe that one also has a logarithmic Euler exact sequence in this situation: 
	$$
	0\rightarrow  \Omega_{\pb^N}(\log D')\rightarrow \tilde{\Omega}_{\pb^N}(\log D')\rightarrow \oc_{\pb^N}\rightarrow 0,
	$$
where we denote	\[\tilde{\Omega}_{\pb^N}(\log D'):=\oc_{\pb^N}(-1)^{\oplus n+1}\oplus \oc_{\pb^N}^{\oplus c}.\]
This logarithmic Euler exact sequence is induced by the usual Euler exact sequence by observing that one has a morphism of log pairs  \(p:(\cb^{N+1}\setminus \{0\},\tilde{D}')\to (\P^N,D')\), where \(p:\cb^{N+1}\setminus \{0\}\to \P^N\) denotes the canonical projection,   \(\tilde{D}'_i:=p^*D'_i\) for any \(i\in \{1,\dots,c\}\), and  \(\tilde{D}'=\sum_{i=1}^c\tilde{D}'_i\). 
It then suffices to observe that the inclusion \(\Omega_{\cb^{N+1}\setminus \{0\}}\hookrightarrow \Omega_{\cb^{N+1}\setminus \{0\}}(\log \tilde{D}')\) descends on \(\P^N\) to the morphism \(\O_{\pb^N}^{\oplus N+1}\to\oc_{\pb^N}^{\oplus n+1}\oplus \oc_{\pb^N}^{\oplus c}(1)\) given by \((\xi_0,\dots, \xi_N)\mapsto (\xi_0,\dots, \xi_n,z_{n+1}\xi_{n+1},\dots, z_{N}\xi_{N})\) and to twist this map by \(\O_{\P^N}(-1)\) in order to make it fit in the following diagram
\[\xymatrix{
	0\ar[r] & \Omega_{\pb^N}\ar[r]\ar[d]& \oc_{\P^N}^{\oplus N+1}(-1)\ar[r]\ar[d] &\oc_{\pb^N}\ar[r]\ar[d]& 0\\
	0\ar[r] & \Omega_{\pb^N}(\log D')\ar[r]& \tilde{\Omega}_{\pb^N}(\log D')\ar[r] &\oc_{\pb^N}\ar[r]& 0.
}
\]

	Let us denote by $N^*=\bigoplus_{i=1}^c\O_{X}(-d_i)$ the conormal bundle of $X$ in $\pb^N$. Since $(X,E)$ is a sub-log manifold of $(\pb^N,D')$, one has the following exact sequence:
	$$
	0\rightarrow N^*\rightarrow \Omega_{\pb^N}(\log D')|_X\rightarrow \Omega_X(\log E)\rightarrow 0,
	$$
	which induces a locally free sheaf \(\tilde{\Omega}_X(\log E)\) sitting in the following commutative diagram
	\[
	\begin{tikzcd} 
	&		& 0 \arrow{d}
	& 0\arrow{d} &  &  &  \\
	&  & N^*\arrow{r}{=}\arrow{d}   & N^*\arrow{d}  	& &  & \\
	& 0 \arrow{r} 
	& \Omega_{\pb^N}(\log D')|_{ X} \arrow{r}\arrow{d}
	& \tilde{\Omega}_{\pb^N}(\log D')|_{ X}\arrow{r}\arrow{d}
	& \oc_X \arrow{r}\arrow{d}
	& 0
	& \\
	& 0\arrow{r}
	& \Omega_{X}(\log E)\arrow{r}\arrow{d}
	&  \tilde{\Omega}_X(\log E)\arrow{r}\arrow{d}
	& \oc_X\arrow{r}
	& 0\\
	& & 0 &\  0. &  
	\end{tikzcd}
	\]
	We obtain an inclusion 
	\[H^0(X,S^m\Omega_X(\log E)\otimes \O_X(-1))\subset H^0(X,S^m\tilde{\Omega}_X(\log E)\otimes \O_X(-1))\]
	and we are therefore reduced to prove that \(H^0(X,S^m\tilde{\Omega}_X(\log E)\otimes \O_X(-1))\) vanishes.

Taking symmetric powers of the middle vertical exact sequence in the above diagram, and twisting it by \(\O_{\P^N}(-1)\) we obtain a resolution 
\[0\to \mathscr{E}_m\to \cdots \to \mathscr{E}_1\to \mathscr{E}_0\to S^m\tilde{\Omega}_X(\log E)\otimes \O_{X}(-1)\to 0,\]
 where
 \[\mathscr{E}_i:=\Lambda^iN^*\otimes S^{m-i}\tilde{\Omega}_{\P^N}(\log D')\otimes \O_{\P^N}(-1).\]
%
Therefore the cohomology of \(S^m\tilde{\Omega}_X(\log E)\otimes \O_X(-1)\) can be computed be the hypercohomology spectral sequence \((E_r^{p,q})\) of the complex \(\mathscr{E}_\bullet\):
\[E_1^{p,q}=H^q(X,\mathscr{E}_{-p})\Rightarrow  H^{p+q}\big(X,S^m\tilde{\Omega}_X(\log E)\otimes \O_X(-1)\big).\]
 Observe that since \(N^{*}\) is of rank \(c\), one has \(\mathscr{E}_{p}=0\) for all \(p> c\).  Moreover, for any \(p\geqslant 0\), the vector bundle \(\mathscr{E}_p\) is a sum of duals of ample  line bundles. Therefore, by Kodaira's vanishing theorem, we obtain \(H^q(X,\mathscr{E}_p)=0\) for any \(q<n\). But  since \(c<n\), this implies that \(E_{1}^{p,q}=0\) if \(p+q= 0\), and therefore we obtain the desired vanishing.
\end{proof}
By contrast, one has the following result. This property is well known (see for instance \cite{Nog81}), but we provide here a short proof for the reader's convenience.
\begin{proposition}
Let \(D=\sum_{i=1}^c D_i\) be a simple normal crossing divisor in \(\P^n\). Then
\[h^0(\P^n,\Omega_{\P^n}(\log D))=c-1\]
\end{proposition} 
\begin{proof}
From the residue  exact sequence  and the vanishing \(H^0(\P^n,\Omega_{\P^n})=0\) we get an exact sequence
\[0\to H^0(\P^n,\Omega_{\P^n}(\log D))\to \bigoplus_{i=1}^c H^0(D_i,\O_{D_i})\cong \cb^c\stackrel{\delta}{\to} H^1(\P^n,\Omega_{\P^n}).\]
Denoting \(e_i\in H^0(D_i,\O_{D_i})\) the constant section equal to \(1\), we can compute \(\delta(e_i)\) by a diagram chase in Cech cohomology and obtain that 
\[\delta(e_i)=c_1(\O_{\P^n}(D_i))=\deg D_i \cdot c_1(\O_{\P^n}(1)).\]
In particular \(\rank (\delta)=1\) and therefore \(h^0(\P^n,\Omega_{\P^n}(\log D))=c-1\).
\end{proof}

\section{A resolution algorithm}\label{se:Resolution} 
In this section we give a proof of \cref{prop:ResolutionAlgo}. To do so we will provide an explicit resolution algorithm for some particular configurations of subvarieties in any ambient complex manifold. 
\subsection{Simple ideal sheaves}\label{sec:simple sheaves}
Let us start by a resolution procedure for some special configurations of linear subspaces in an affine space. This is purely algebraic and could be performed over any field.
Fix an integer \(m\geqslant 2\) and let $\A^{m}$ be the dimension \(m\) affine space over \(\cb\). We say that an ideal  \(J\subset \cb[x_1,\dots,x_m]\) is simple if it is of the form
	\begin{eqnarray}\label{simple sheaf}
	J=\langle x_1,\ldots,x_p,x_{p+1}x_{r+1},x_{p+2}x_{r+2},\ldots,x_{r}x_{2r-p}\rangle.
	\end{eqnarray}
Since $J=\sqrt{J}$, one has $J=I\big(\mathbb{V}(J)\big)$, where $\mathbb{V}(J)$ is the (not necessarily irreducible) variety defined by $J$. In fact, one has 
	 \[J=\bigcap_{K\subset \{1,\dots, r-p\}} \langle x_1,\ldots,x_p,x_{p+k_1},\dots x_{p+k_s},x_{r+\ell_1},\dots,x_{r+\ell_{r-p-s}}\rangle\]
	 Where \(K=(k_1,\dots, k_s)\) and \(\{1,\dots, r-p\}\setminus K=\{\ell_1,\dots, \ell_{r-p-s}\}\).
	 We say that a subvariety of \(\ab^m\) is simple if its defining ideal is of the form \(\langle x_{i_1},\dots,x_{i_n}\rangle\) for some \(\{i_1,\dots, i_n\}\subset \{1,\dots, m\}\).
Therefore the above relation shows that \(\mathbb{V}(J)=\bigcup_{i=1}^NV_i\) is the union of simple varieties of same dimensions.
\begin{lem}\label{blowing-up}
Take a simple ideal \(J\subset\cb[x_1,\dots, x_m] \) as \eqref{simple sheaf} and write  \(\mathbb{V}(J)=\bigcup_{i=1}^NV_i\)  as above. Let \(i\in \{1,\dots,m\}\). Let $\mu:\widetilde{X}\rightarrow \ab^m$ be the blow-up of $V_i$. Let us write  $E_i:=\mu^{-1}(V_i)$ the exceptional divisor and $\tilde{V}_j$ the strict transform of $V_j$ for every $j\neq i$.  Then $\mu^*J=\oc_{\widetilde{X}}(-E_i)\cdot J_i$, where $J_i$ denotes the ideal sheaf of the variety $\bigcup_{j\neq i}\tilde{V}_j$. Moreover, one can cover $\tilde{X}$ by open sets which are all isomorphic to \(\A^m\) such that on every such open set, the ideal $J_i$ is either a simple ideal or the trivial ideal \(\O(\A^m)\).
\end{lem}
\begin{proof}
	Without loss of generality, we can take $V_1=\{(x_1,\ldots,x_m)\mid x_1=x_{2}=\ldots=x_{r}=0 \}$, and it suffices to prove the lemma for the  the blow-up $\mu:\widetilde{X}\rightarrow \ab^m$ of \(V_1\).
	
	In the $x_1$-direction, one can take an affine  open set \(X_1:=\spec\cb[x_1,\tilde{x}_2,\ldots,\tilde{x}_r,x_{r+1},\ldots,x_m] \) of $\widetilde{X}$, such that the blow-up $\mu$ is given by $\mu(x_1,\tilde{x}_2,\ldots,\tilde{x}_r,x_{r+1},\ldots,x_{m})=(x_1,x_1\tilde{x}_2,\ldots,x_1\tilde{x}_r,x_{r+1},\ldots,x_{m})$, and thus 
	$$
	\mu^*J|_{X_1}= x_1\cdot \langle 1,\tilde{x}_2,\ldots,\tilde{x}_p,\tilde{x}_{p+1}x_{r+1},\tilde{x}_{p+2}x_{r+2},\ldots,\tilde{x}_{r}x_{2r-p}\rangle=\langle x_1\rangle,
	$$
	where $\{x_1=0\}$ defines the exceptional divisor $E_1$ on $X_1$. We then observe that \[J_1|_{X_1}=\cb[x_1,\tilde{x}_2,\ldots,\tilde{x}_r,x_{r+1},\ldots,x_m],\]
	which is the trivial ideal. 
	This implies that \(\tilde{V}_{i}
\bigcap X_1=\varnothing \) for each \(i=2,\ldots,N\), which also holds on the affine open set \(X_j\) in the $x_j$-direction with $j=2,\ldots,p$ by \eqref{simple sheaf}.
	
	In the $x_{r}$ direction, one can take an affine  open set \(X_r:=\spec\cb[\tilde{x}_1,\ldots,\tilde{x}_{r-1},x_{r},\ldots,x_m] \) of $\widetilde{X}$, such that the blow-up $\mu$ is given by $\mu(\tilde{x}_1,\tilde{x}_2,\ldots,\tilde{x}_{r-1},x_{r},\ldots,x_{m})=(\tilde{x}_1x_r,\tilde{x}_2x_r,\ldots,\tilde{x}_{r-1}x_r,x_{r},\ldots,x_{m})$. We thus see that
	$$
	\mu^*J|_{X_r}= x_r\cdot \langle \tilde{x}_1,\tilde{x}_2,\ldots,\tilde{x}_p,\tilde{x}_{p+1}x_{r+1},\ldots,\tilde{x}_{r-1}x_{2r-p-1},\tilde{x}_{2r-p}\rangle,
	$$
	where $\{x_r=0\}$ defines the exceptional divisor $E_1$, and we see that 
	$$J_1|_{X_r}=\langle \tilde{x}_1,\tilde{x}_2,\ldots,\tilde{x}_p,\tilde{x}_{p+1}x_{r+1},\ldots,\tilde{x}_{r-1}x_{2r-p-1},\tilde{x}_{2r-p}\rangle.$$
Here \( \mathbb{V}(J_1|_{X_r})= \bigcup_{1<j\leq N}\tilde{V}_j|_{X_r}\)	Thus $J_1|_{X_r}$ is still a simple ideal sheaf, and $\tilde{V}_j|_{X_r}$ is either empty or a simple variety for any \(j\neq 1\). This also holds on the affine open sets \(X_\ell\) in the $x_\ell$ directions with $\ell=p+1,\ldots,{r-1}$. Since \(\bigcup_{i=1}^{r}X_i=\tilde{X} \), this proves the lemma.
\end{proof}

The following lemma will be useful later.
\begin{lem}\label{disjoint}
	Let $\{V_i\}_{i=0,1,2,3}\subset \A^m$ be different simple varieties of the same dimension. Assume that $V_1\cap V_2\subset V_3$. Let $\widetilde{X}\rightarrow \ab^m$ be the blowing-up of $V_0$, with $\tilde{V}_1,\tilde{V}_2,\tilde{V}_3$ the strict transforms of \(V_1,V_2,V_3\) and $E$ the exceptional divisor. Then $\tilde{V}_1\cap\tilde{V}_2\subset\tilde{V}_3$. Moreover, if $V_1\cap V_2\subset V_0$, then $\tilde{V}_1\cap \tilde{V}_2=\varnothing$.
\end{lem}
\begin{proof}
	Since $\mu:\widetilde{X}\setminus E\rightarrow \ab^m\setminus V_0$ is an isomorphism, it follows from $V_1\cap V_2\subset V_3$ that
	\[\big(\tilde{V}_1\cap \tilde{V}_2\setminus E\big)\subset \tilde{V}_3.\]
		Note that $E$  can naturally be seen as the projectivization of the normal bundle $N_{V_0/\ab^m}$ of $V_0$ in $\ab^m$. Then for any $\tilde{x}\in E$, one can write $\tilde{x}=(x,[v])$ with $x\in V_0$ and $v\in N_{V_0/\ab^m,x}$. If $\tilde{x}\in 	\tilde{V}_1\cap \tilde{V}_2\cap E$, one has $x\in V_1\cap V_2\cap V_0$, and $v\in T_{V_1,x}\bigcap T_{V_2,x}$. Since $V_1$ and $V_2$ are simple varieties, by the definition, one has $T_{V_1,x}\cap T_{V_2,x}=T_{V_1\cap V_2,x}\subset T_{V_3,x}$. We then conclude that $\tilde{x}\in \tilde{V}_3\cap E$, and $\tilde{V}_1\cap \tilde{V}_2\subset \tilde{V}_3$. This proves the first claim. When $V_1\cap V_2\subset V_0$, for any $x\in V_1\cap V_2$, one has $T_{V_1\cap V_2,x}\subset T_{V_0,x}$, and thus $\tilde{V}_1\cap \tilde{V}_2\cap E=\varnothing$. Since $\tilde{V}_1\cap \tilde{V}_2\setminus E\simeq V_1\cap V_2\setminus V_0=\varnothing$, one has $\tilde{V}_1\cap \tilde{V}_2=\varnothing$. This proves the second claim, hence the lemma.
\end{proof}

\subsection{Resolution of compatible systems}

\begin{dfn}\label{compatible}
Let \(X\) be a smooth (not necessarily compact) complex manifold of dimension \(m\). A finite collection of subvarieties of \(X\) is said to be a \emph{compatible system} if this collection admits an indexation of the form $\{Y_{ij}\}_{i=a,\ldots,c}^{j=1,\ldots,n_i}$ such that, denoting by  \(\js \)  the ideal sheaf of  $\bigcup_{i=a,\ldots,c}^{j=1,\ldots,n_i}Y_{ij}$,   the following conditions are satisfied.
	\begin{thmlist}
		\item \label{1} For any $j\neq \ell$, either there exists $i'<i$ and $j'$ such that $Y_{ij}\cap Y_{i\ell}\subset Y_{i'j'}$ or $Y_{ij}\cap Y_{i\ell}=\varnothing$; in particular,  $Y_{aj}\cap Y_{a\ell}=\varnothing$.
		\item \label{2} The manifold \(X\) can be covered by analytic open sets $\{U_\alpha\}$, such that  each \(U_\alpha\) is biholomorphic to some open set \(W_\alpha\subset \cb^m \) containing the \(0\in \cb^m \)  and satisfying the following. Let \(J_{i,\alpha}:=\{j\mid Y_{ij}\cap U_\alpha\neq \varnothing  \} \), then via the isomorphism \(U_\alpha\cong W_\alpha\subset \cb^m \), \(\{Y_{ij}\cap U_\alpha\}_{i=a,\ldots,c}^{j\in J_{i,\alpha}}\) extends to \(\{Y_{ij}^\alpha\}_{i=a,\ldots,c}^{j\in J_{i,\alpha}} \)	in \(\cb^m\), such that \(Y_{ij}^\alpha\) is a  simple variety and the ideal sheaf \(\js_\alpha\) defined by \(\bigcup_{i=a,\ldots,c}^{j\in J_{i,\alpha}}Y_{ij}^{\alpha} \) is a simple ideal of the form  \eqref{simple sheaf}. 
	\end{thmlist}
We say that  $i$ is the \emph{index} of $Y_{ij}$, that  $a$ is the \emph{lowest index},  that  $c-a+1$ is the \emph{length}  of this compatible system, and that \(\js\) the \emph{ideal sheaf associated to this system}.
\end{dfn}
Based on the \cref{compatible}, we have the following observation.
\begin{lem}\label{lem:observation} With the notation of \cref{compatible}.
	If \(Y_{ij}\bigcap Y_{i'j'}\subset Y_{k\ell} \), then for each \(\alpha\), one has \(Y^\alpha_{ij}\bigcap Y^\alpha_{i'j'}\subset Y^\alpha_{k\ell} \). In particular, \(\#J_{a,\alpha}\leqslant 1 \).
\end{lem}
\begin{proof}
	By  assumption, for any \(\alpha\), one has
	\((Y_{ij}\bigcap U_\alpha)\bigcap (Y_{i'j'}\bigcap U_\alpha)\subset Y_{k\ell}\bigcap U_\alpha \). By \cref{2}, this is equivalent to \((Y_{ij}^\alpha\bigcap W_\alpha)\bigcap(Y^\alpha_{i'j'}\bigcap W_\alpha)\subset Y^\alpha_{k\ell}\bigcap W_\alpha \). Since \(W_\alpha\) contains the origin, and \(\js_{\alpha} \) is a simple ideal, from the very definition \eqref{simple sheaf} we observes that 
	\[Y_{ij}^\alpha \bigcap Y^\alpha_{i'j'} \subset Y^\alpha_{k\ell}.
	\]
	For any \(j\in J_{i,\alpha} \), \(Y_{ij}^\alpha\) contains the origin, and by \cref{1} one has \(\#J_{i,\alpha}\leqslant 1 \).
\end{proof}

Now we  define a natural log resolution of the  ideal sheaf $\js$. We first define $\mu_1:\widetilde{X}_1\rightarrow X$ to be the blow-up of $\bigcup_{j=1}^{n_a}Y_{aj}$, with $E_1$ the exceptional divisor and $\tilde{Y}_{ij}$ the strict transform of $Y_{ij}$ for $i>a$. By  \cref{1}, $\bigcup_{j=1}^{n_a}Y_{aj}$ is a disjoint union of smooth submanifold of the same dimension, and thus \(\tilde{X}_1\) is also a smooth manifold. Write \(\tilde{Y}_{ij}\) for the strict transform of \(Y_{ij}\) under the blow-up \(\mu_1\) for any \(i>a\), and \(j=1,\ldots,n_i \). Denote by $\js_1$  the ideal sheaf of the variety $\bigcup_{i=a+1,\ldots,c}^{j=1,\ldots,n_i}\tilde{Y}_{ij}$ in \(\tilde{X}_1 \).
\begin{lem}\label{lem:new system}
 If the index of \(\tilde{Y}_{ij}\) is defined to be \(i\) for any  \(i\in \{a+1,\ldots,c\}\) and any  \(j\in\{1,\ldots,n_i\}\), then family $\{\tilde{Y}_{ij}\}_{i=a+1,\ldots,c}^{j=1,\ldots,n_i}$ is a compatible system in \(\tilde{X}_1 \). Moreover
\begin{eqnarray}\label{eq:inverse}
\mu_1^*\js=\oc_{\widetilde{X}_1}(-E_1)\cdot \js_1.
\end{eqnarray}
\end{lem}
\begin{proof}
 On each \(U_\alpha \), if \(J_{i,\alpha}\neq\varnothing\), define \(\mu_{1,\alpha}:Z_\alpha^1\rightarrow \cb^m \) to be the blow-up of \(\bigcup_{j\in J_{a,\alpha}}Y_{aj}^\alpha \),  let  \(E_{1,\alpha}\) be exceptional divisor and for any \(i>a\) and \(j\in J_{i,\alpha} \) let \(\tilde{Y}_{ij}^\alpha \) be the strict transform of \({Y}_{ij}^\alpha\). If \(J_{i,\alpha}=\varnothing\), set \(\mu_{1,\alpha}\) to be the identity map. Then via the isomorphism \(U_\alpha\cong W_\alpha \), one has \(\js_\alpha|_{W_\alpha}\cong \js|_{U_\alpha} \),
 \begin{eqnarray}\label{eq:isomorphism}
 \tilde{Y}_{ij}^\alpha|_{\mu_{1,\alpha}^{-1}(W_\alpha)}\cong \tilde{Y}_{ij}|_{\mu_{1}^{-1}(U_\alpha)}
 \end{eqnarray} 
 and by \cref{2} one has the following isomorphism
\begin{eqnarray}\label{local correspondance}
\xymatrix
{ \mu_{1,\alpha}^{-1}(W_\alpha)\ar[r]^-{\mu_{1,\alpha}} \ar[d]_-{\cong} &W_\alpha\ar[d]_-{\cong}\\
	\mu_{1}^{-1}(U_\alpha)\ar[r]^-{\mu_{1}}  &U_\alpha.
}
\end{eqnarray}
 It follows from \Cref{blowing-up} that 
\[
\mu_{1,\alpha}^*\js_\alpha=\oc_{Z_\alpha^1}(-E_{1,\alpha})\cdot \js_{1,\alpha},
\]
where $\js_{1,\alpha}$ denotes the ideal sheaf of the variety $\bigcup_{i=a+1,\ldots,c}^{j\in J_{i,\alpha}}\tilde{Y}^\alpha_{ij}$.  By \eqref{local correspondance}, one has
\begin{eqnarray} \label{eq:local inverse}
\mu_1^*\js|_{U_\alpha}=\oc_{\mu_{1}^{-1}(U_\alpha)}(-E_1)\cdot \js_1|_{\mu_{1}^{-1}(U_\alpha)}.
\end{eqnarray}
and \eqref{eq:inverse} follows from that  \(U_\alpha \) is an open covering for \(X\).

For any \(i>a \) any $j\neq l$, by \cref{1} either $Y_{ij}\cap Y_{i\ell}=\varnothing$, in which one also has
\begin{eqnarray}\label{eq:empty intersection}
\tilde{Y}_{ij}\cap \tilde{Y}_{i\ell}=\varnothing;
\end{eqnarray}
or there exists $i'<i$ and $j'$ such that $Y_{ij}\cap Y_{i\ell}\subset Y_{i'j'}$. In the second case, we claim that  $\tilde{Y}_{ij}\cap \tilde{Y}_{i\ell}\subset \tilde{Y}_{i'j'}$ when \(i'>a\), and \eqref{eq:empty intersection} holds when \(i'=a\). since \(U_\alpha\) is a covering of \(X\) and since the \(Y_{ij}\)'s are irreducible, it suffices to check this on each \(\mu_1^{-1}(U_\alpha)\). By \Cref{lem:observation}, \(Y^\alpha_{ij}\cap Y^\alpha_{i\ell}\subset Y^\alpha_{i'j'}\), and from \cref{2} they are all simple varieties in \(\cb^m\). Applying \Cref{disjoint}, one has
\[\tilde{Y}^\alpha_{ij}\bigcap \tilde{Y}^\alpha_{i\ell}\subset 
\begin{cases}
  \tilde{Y}^\alpha_{i'j'} & \mbox{ if } i'>a  \mbox{ and } j'\in J_{i'\alpha} \\
\varnothing & \mbox{ if } i'=a \mbox{ or } j\notin J_{i'\alpha} 
\end{cases}
\]
The claim then follows immediately from \eqref{eq:isomorphism}. The set  $\{\tilde{Y}_{ij}\}_{i=a+1,\ldots,c}^{j=1,\ldots,n_i}$ satisfies \cref{1}.

For any \(Z_{1,\alpha}\), by the proof of \Cref{blowing-up} it can be covered by finitely many open sets which are all isomorphic to affine spaces \( \spec \cb[x_1,\ldots,x_m] \), such that the restriction of  \(\js_{1,\alpha}\) to each affine space is still a simple ideal. Since \(W_\alpha\) contains the origin, the restriction of  \(\mu_{1,\alpha}^{-1}(W_\alpha) \) to each affine space still contains the origin. From \eqref{eq:isomorphism}, this gives an open covering which verifies \cref{2}. We define \(i \) to be the index of \(\tilde{Y}_{ij}\) for any \(i=a+1,\ldots,c\) and any \(j=1,\ldots,n_i\). $\{\tilde{Y}_{ij}\}_{i=a+1,\ldots,c}^{j=1,\ldots,n_i}$ is then a compatible system in \(\tilde{X}_1 \).
\end{proof}
Observe that after the blow-up \(\mu_1:\tilde{X}_1\rightarrow X\), and with our choice of indices, the length of the compatible system is decreased by one. We then blow up $\bigcup_{j=1}^{n_{a+1}}\tilde{Y}_{a+1,j}$ to obtain $\mu_2:\widetilde{X}_2\rightarrow \widetilde{X}_1$, and continuing in this way, we obtain a sequence of blow-ups with smooth centers 
\begin{eqnarray}\label{sequence}
\widetilde{X}_{c-a+1}\xrightarrow{\mu_{c-a+1}}\widetilde{X}_{c-a}\xrightarrow{\mu_{c-a}} \cdots \xrightarrow{\mu_3}\widetilde{X}_2\xrightarrow{\mu_2} \widetilde{X}_1\xrightarrow{\mu_1} X,
\end{eqnarray}
such that  their composition $\mu:\widehat{X}:=\widetilde{X}_{c-a+1}\rightarrow X$ gives rise to a  log resolution  of $\js$. Indeed, this algorithm is blowing-up the varieties of \emph{lowest index} in each newly obtained compatible system. Since after each blow-up, the inverse image of the ideal sheaf is, up to some exceptional divisors, the ideal sheaf associates to a compatible system whose  length  is decreased by 1, we see that  after  $(c-a+1)$-blow-ups, one resolves the ideal sheaf $\js$. We say that $\mu:\widetilde{X}_{c-a+1}\rightarrow X$ the \emph{canonical log resolution} of the ideal sheaf $\js$ associated to the compatible system  $\{{Y}_{ij}\}_{i=a,\ldots,c}^{j=1,\ldots,n_i}$ in $X$.  We will now prove that this log resolution also resolves  certain subsheaves of $\js$ simultaneously. 
\begin{dfn}\label{sub-compatible}
Same notation as \cref{compatible}.  We say  that $\{Y_{ij}\}_{i=a,\ldots,b}^{j=1,\ldots,r_i}$ with $b\leqslant c $ and $r_i\leqslant n_i$ is a \emph{subsystem} of $\{Y_{ij}\}_{i=a,\ldots,c}^{j=1,\ldots,n_i}$, if it is also a compatible  system in $X$, and satisfies that 
  for any $i=a,\ldots,b$, $j=r_i+1,\ldots,n_i$, $k=i,i+1,\ldots,b$, and $\ell=1,\ldots,r_k$, either $Y_{ij}\cap Y_{k\ell}=\varnothing$, or there exists $p<i$ and $1\leqslant q\leqslant r_p$ such that
	\begin{eqnarray}\label{condition}
	Y_{ij}\cap Y_{k\ell}\subset Y_{pq}.
	\end{eqnarray}
	In particular, one has
 \begin{eqnarray}\label{independent}
 \big(\bigcup_{j=r_{a}+1,\ldots,n_a}Y_{aj}\big)\bigcap \big(\bigcup_{\substack{i=a,\ldots,b\\ j=1,\ldots,r_i}}Y_{ij}\big)=\varnothing.
 \end{eqnarray}
\end{dfn}
\begin{lem}\label{simultaneous}
Let $\{Y_{ij}\}_{i=a,\ldots,c}^{j=1,\ldots,n_i}$ be a compatible system on a quasi-projective variety \(X\), and let $\{Y_{ij}\}_{i=a,\ldots,b}^{j=1,\ldots,r_i}$ be a subsystem.   Let  $\js'$ be the ideal sheaf of  $\bigcup_{i=a,\ldots,b}^{j=1,\ldots,r_i} Y_{ij}$, then $\mu:\hat{X}\rightarrow X$ is also a log resolution for $\js'$.
\end{lem}

\begin{proof}
	Since $\{Y_{aj}\}_{j=1,\ldots,n_a}$ are pairwise disjoint, it follows from \eqref{independent} and \cref{lem:new system} that for the blow-up $\mu_1:\widetilde{X}_1\rightarrow X$  of  $\bigcup_{j=1}^{n_a}Y_{aj}$, one has 
	$$
	\mu_1^*\js'=\oc_{\widetilde{X}}(-E_1')\cdot \js'_1,
	$$
	where $E_1'$ is the inverse image of $\bigcup_{j=1}^{r_a}Y_{aj}$ and $\js'_1$ is the sheaf of the variety $\bigcup_{i=a+1,\ldots,b}^{j=1,\ldots,r_i}\tilde{Y}_{ij}$. By \cref{lem:new system} again we note that  $\{\tilde{Y}_{ij}\}_{i=a+1,\ldots,b}^{j=1,\ldots,r_i}$ is also a compatible system in $\tilde{X}_1$. We will prove it is also a subsystem of   $\{\tilde{Y}_{ij}\}_{i=a+1,\ldots,c}^{j=1,\ldots,n_i}$.

For any $i=a+1,\ldots,b$, $j=r_i+1,\ldots,n_i$, $k=i,i+1,\ldots,b$, and $\ell=1,\ldots,r_k$,  by \cref{sub-compatible} either $Y_{ij}\cap Y_{k\ell}=\varnothing$, or there exists $p<i$ and $1\leqslant q\leqslant r_p$ such that \(Y_{ij}\cap Y_{k\ell}\subset Y_{pq}\). As in the proof of \Cref{lem:new system}, we can prove that 
\[\tilde{Y}_{ij}\bigcap \tilde{Y}_{k\ell}\subset 
\begin{cases}
\tilde{Y}_{pq} & \mbox{ if } p>a \\
\varnothing & \mbox{ if } p=a \mbox{ or } i=a+1
\end{cases}
\]
This verifies the conditions in \cref{sub-compatible}. Thus   $\{\tilde{Y}_{ij}\}_{i=a+1,\ldots,b}^{j=1,\ldots,r_i}$ is    a subsystem of  the compatible system  $\{\tilde{Y}_{ij}\}_{i=a+1,\ldots,c}^{j=1,\ldots,n_i}$. We conclude by
induction on the length of the system of compatible subvarieties that for the composition of the blow-ups 
$$
\tilde{X}_{b-a+1}\xrightarrow{\mu_{b-a+1}}\tilde{X}_{b-a}\xrightarrow{\mu_{b-a}} \cdots \xrightarrow{\mu_2} \widetilde{X}_1\xrightarrow{\mu_1} X,
$$
it resolves the ideal sheaf $\js'$. This proves the lemma.
\end{proof}
\subsection{Functorial properties}

This resolution algorithm  is \emph{functorial} under restrictions.
\begin{lem}\label{functorial}
	Let $Z$ be a regular subvariety of $X$. Assume that $\{Y_{ij}\cap Z\}_{i=a,\ldots,c}^{j=1,\ldots,n_i}$ is also a compatible system in $Z$ if we define the index of $Y_{ij}\cap Z$ to be $i$ for each $Y_{ij}$. We emphasize here that we consider the scheme theoretic intersection \(Z\cap Y_{ij}\). Denote by $\widetilde{Z}$ the strict transform of $Z$ under $\mu$. Then the restriction of the canonical log resolution $\mu$ associated to   $\{{Y}_{ij}\}_{i=a,\ldots,c}^{j=1,\ldots,n_i}$ to $\widetilde{Z}$ is also the canonical log resolution associated to  the compatible system $\{{Y}_{ij}\cap Z\}_{i=a,\ldots,c}^{j=1,\ldots,n_i}$ in $Z$.  In particular, $\mu|_{\widetilde{Z}}$ is the log resolution of the ideal sheaf of the subvariety $\bigcup_{i=a,\ldots,b}^{j=1,\ldots,n_i}({Y}_{ij}\cap Z)$.
\end{lem}
\begin{proof}
By induction on the length of  the compatible system. Assume that the lemma holds when the length is less than $c-a$. Since $\mu_1:\widetilde{X}_1\rightarrow X$ is the blow-up of the variety $\bigcup_{ j=1,\ldots,n_a}{Y}_{aj}$, we deduce that  $\mu_1| _{\widetilde{Z}_1}:\widetilde{Z}_1\rightarrow Z$  is the blow-up of the variety $\bigcup_{ j=1,\ldots,n_a}({Y}_{aj}\cap Z)$, where $\widetilde{Z}_1$ denotes the strict transform of $Z$ under $\mu_1$. By our definition of   the indices of the subvarieties \(Z\cap Y_{ij}\), $\mu_1|_{\widetilde{Z}_1}$ is the first blow-up of the canonical log resolution associated to  the compatible system $\{{Y}_{ij}\cap Z\}_{i=a,\ldots,c}^{j=1,\ldots,n_i}$ in $Z$, such that the strict transform of ${Y}_{ij}\cap Z$ is $\tilde{Y}_{ij}\cap \widetilde{Z}_1$. Then $\{\tilde{Y}_{ij}\cap \widetilde{Z}_1\}_{i=a+1,\ldots,c}^{j=1,\ldots,n_i}$ is a compatible system in $\widetilde{Z}_1$,  with the index $\tilde{Y}_{ij}\cap \widetilde{Z}_1$ equal to $i$. By the induction on the length, the canonical log resolution   associated to the compatible system $\{\tilde{Y}_{ij}\}_{i=a+1,\ldots,c}^{j=1,\ldots,n_i}$ in $\widetilde{X}_1$, which is the composition of the blow-ups,  $\tilde{X}_{b-a+1}\xrightarrow{\mu_{b-a+1}}\tilde{X}_{b-a}\xrightarrow{\mu_{b-a}} \cdots \xrightarrow{\mu_2} \widetilde{X}_1$, gives rise  to   the canonical log resolution associated to the compatible system $\{\tilde{Y}_{ij}\cap \widetilde{Z}_1\}_{i=a+1,\ldots,c}^{j=1,\ldots,n_i}$ in $\widetilde{Z}_1$. The lemma is thus proved.
\end{proof}
This resolution algorithm is also compatible with families. 
\begin{dfn}
	Let \(f:\xs\rightarrow S\) be a surjective smooth morphism between regular quasi-projective varieties, and let \(\ys\)  be (non necessarily irreducible) subvariety of \(\xs\) of pure dimension \(r\) such that its image under \(f\) is surjective.  We say that the family \(f:(\xs,\ys)\rightarrow S \)  is \emph{locally analytically trivial} if for any \(y\in \ys\), setting   \(s:=f(y)\) and \(X_s\) the fiber of \(f\), the following two conditions hold:
	\begin{thmlist}
		\item \label{def1} There exists  \(\mathscr{U}\subset \mathscr{X},U_s\subset S\) open neighborhoods of \(y,s\), such that \(f(\mathscr{U})=U_s\).  
		\item \label{def2} Write \(U:=\mathscr{U}\cap{X_s} \), which is an analytic open set of \(X_s\).  There exists a biholomorphism \(\theta:\mathscr{U}\to U\times U_s\) such that the following diagram is commutative:
\begin{eqnarray}\label{local trivial}
\xymatrix{
	\mathscr{U}\ar[r]^-{\theta}_-{\cong}\ar[dr]_{f}&  U\times U_s \ar[d]^{{\rm pr}_2}\ar[r]^-{{\rm pr}_1}&U\\
	&U_s
}
\end{eqnarray}
such that \((\mathcal{U},\ys\cap{U})\cong \big(U\times U_s,{\rm pr}_1^*(Y_U)\big)\) via \(\theta\), where \(Y_U:=\ys\cap U \).
	\end{thmlist}
\end{dfn}

\begin{lem}\label{lem:resolution family}
Let \(f:(\xs,\ys)\rightarrow S \)  be a locally analytically trivial family. Suppose that \(\ys= \bigcup_{i=a,\dots,c}^{j=1\dots, n_i}\ys_{ij} \) such that for any \(s\in S\), setting \(Y_{s,ij}=\ys_{ij}|_{X_s} \), the induced configuration \(\{Y_{s,ij}\}_{i=a,\dots,c}^{j=1\dots, n_i}\) is a compatible system on \(X_s\). Then we can perform the algorithm of canonical resolution  defined in \eqref{sequence}   for \(\{\ys_{ij}\}_{i=a,\dots,c}^{j=1\dots, n_i}\), in order to  obtain a birational map \(\mu_{\X}:\widetilde{\X}\to \X\), such that, for any \(s\in S\), the restriction
\[
\mu_\xs|_{X_s}:\mu_{\X}^{-1}(X_s)\rightarrow X_s
\] 
is  the canonical resolutions of the compatible system \(\{Y_{s,ij}\}_{i=a,\dots,c}^{j=1\dots, n_i}\), denoted by  \(\mu_s:\tilde{X}_s\rightarrow X_s \).
If one denotes the irreducible exceptional divisors of \(\mu_\X\) by \(\mathscr{E}_1\dots, \mathscr{E}_r\) then, the set of irreducible exceptional divisor of \(\mu_s\) is \(\big\{\mathscr{E}_1|_{\mu_{\X}^{-1}(X_s)}\dots, \mathscr{E}_r|_{\mu_{\X}^{-1}(X_s)}\big\}\)

\end{lem}
\begin{proof}
We will also prove by induction on the length of  the compatible system. Assume that the lemma holds when the length is less than $c-a$.	We make same notation in \cref{def2}. By the assumption of the lemma,   one observes that \( \{\ys_{aj}\}_{j=1,\ldots,n_a}  \) is disjoint from each other since this holds on each fiber. Define \(\mu_{\xs,1}:\tilde{\xs}_1\rightarrow\xs \) to be the blow-up of \( \{\ys_{aj}\}_{j=1,\ldots,n_a} \), with \(\{\tilde{\ys}_{ij}\}_{i=a+1,\dots,c}^{j=1\dots, n_i} \)   the strict transforms of \(\{{\ys}_{ij}\}_{i=a+1,\dots,c}^{j=1\dots, n_i}\).  By \cref{def2}, its restriction to each fiber \(X_s\) is precisely the first blow-up of the canonical log resolution of the compatible system \(\{Y_{s,ij}\}_{i=a,\dots,c}^{j=1\dots, n_i}\). From \cref{lem:new system}, the sets \(\{\tilde{\ys}_{ij}\}_{i=a+1,\dots,c}^{j=1\dots, n_i} \)  verify the conditions of the lemma with the length decreased by 1. Hence we prove the lemma. 
\end{proof}
\subsection{Application to the projectivized logarithmic cotangent bundle}\label{indeterminacies}

In this section, we prove \cref{prop:ResolutionAlgo}.  From \Cref{cor:MoreComponents} we can assume that \(c\leqslant n\).
\begin{lem} \label{resolution}
Let \((X,D)\) be a log pair. The family 
	$\{\tilde{D}_J\}_{\varnothing\subsetneq J\subset \{1,\ldots,c\}}$ is a compatible system in ${X}_1(D)$ when the index of \(\tilde{D}_J\) is defined to be \(\# J\). Its lowest index is \(1\) and its length  is \(c\). Moreover, for any $\varnothing\neq I\subsetneq \{1,\ldots,c\}$,  $\{\tilde{D}_J\}_{\varnothing\neq J\subset I}$  is a subsystem  of length \(\#I \). In particular, the canonical log resolution  provides a birational map  $\hat{\mu}:\widetilde{X}_1(D)^c\to X_1(D)$  which is  a simultaneous log resolution for the ideal sheaves $\{\js_I\}_{\varnothing\neq I\subset \{1,\ldots,c\}}$.
\end{lem}
\begin{proof}
  For any $\varnothing\neq I,J\subset \{1,\ldots,c\}$ with $1\leqslant \#I=\#J<c$, and $I\neq J$, note that $\#(I\cap J)<\#I=\#J<\#(I\bigcup J)$ and by  \eqref{intersection} one has
\begin{eqnarray}\label{eq:intersection}
\tilde{D}_I\cap \tilde{D}_J\begin{cases}
=\varnothing  & \quad \text{if } I\cap J=\varnothing\\
\subset   \tilde{D}_{I\cap J}   & \quad \text{if } I\cap J\neq\varnothing
\end{cases}
\end{eqnarray} 
where in the second case, $1\leqslant \#(I\cap J)$. In particular, when \(\#I=\#J=1\), we always have \(\tilde{D}_I\cap \tilde{D}_J=\varnothing\). 

Now, for any $\varnothing\neq J\subset \{1,\ldots,c\}$,  we define the index of $\tilde{D}_J$   to be $\#J$.	For any \(i=1,\ldots,c\), set \( S_i:=\{I\subset\{1,\ldots,c \}\ /\ \#I=i\} \), equipped with the \emph{lexicographical ordering}, which is a total order. Then we can write  $\{\tilde{D}_J\}_{\varnothing\neq J\subset \{1,\ldots,c\}}=\{Y_{ij}\}_{i=1,\ldots,c}^{j=1,\ldots,n_i}$ with \(n_i=\binom{c}{i}\) in a canonical way, such that, for any $I,J\in S_i$ with $I< J$, there exists \(1\leqslant a<b\leqslant n_i  \) such that \(D_I=Y_{ia}\) and \(D_J=Y_{ib}\). With this setting, from \eqref{eq:intersection} one can see that, for any $i>1$ and any $1\leqslant j<l\leqslant n_i$, either $Y_{ij}\cap Y_{i\ell}=\varnothing$, or there exists $1\leqslant i'<i$ and $1\leqslant j'\leqslant n_{i'}$ such that $Y_{ij}\cap Y_{i\ell}\subset Y_{i'j'}$. Thus the set \(\{Y_{ij}\}_{i=1,\ldots,c}^{j=1,\ldots,n_i}\) satisfies \cref{1}. Since we can cover \(X_1(D)\) by open sets with \(\js_I\) in the form of  \eqref{eq:lift ideal}, which satisfies \cref{2}, we conclude that $\{\tilde{D}_J\}_{\varnothing\neq J\subset \{1,\ldots,c\}}$ is a compatible system in ${X}_1(D)$. We thus prove the first statement.

To prove the second statement, without loss of generality, we can assume that  $ I= \{1,\ldots,b\}$ for some \(1\leqslant b<c\). From \eqref{intersection} we first observe that \(\{\tilde{D}_J\}_{\varnothing\neq J\subset I}\) satisfies \cref{1}, and thus is a compatible system on \(X_1(D)\). Within the above  representation \(\{\tilde{D}_J\}_{\varnothing\neq J\subset \{1,\ldots,c\}}=\{Y_{ij}\}_{i=1,\ldots,c}^{j=1,\ldots,n_i}\),   one has $\{\tilde{D}_J\}_{\varnothing\neq J\subset I}=\{Y_{ij}\}_{i=1,\ldots,b}^{j=1,\ldots,r_i}$, where \(r_i=\binom{b}{i} \). For any $i=1,\ldots,b$, $j=r_i+1,\ldots,n_i$, $k=i,i+1,\ldots,b$, and $\ell=1,\ldots,r_k$,  let us denote by
$\tilde{D}_J=Y_{ij}$, $\tilde{D}_K=Y_{k\ell}$. If follows from the definition that 
$J\not\subset I$, $\#J=i$, and $K\subset I$, $\#K=k\geqslant i$, and  from \eqref{intersection} that either $\tilde{D}_J\cap \tilde{D}_K=\varnothing$, in which case $J\cap K=\varnothing$, or $\tilde{D}_J\cap \tilde{D}_K=\tilde{D}_{J\cap K}$, in which case $1\leqslant \#(J\cap K)<\#J$. Thus by \Cref{sub-compatible} $\{\tilde{D}_J\}_{\varnothing\neq J\subset I}$ is a subsystem of $\{\tilde{D}_J\}_{\varnothing\neq J\subset \{1,\ldots,c\}}$. This proves the second statement. 

 Since the length of the compatible system
$\{\tilde{D}_J\}_{\varnothing\subsetneq J\subset \{1,\ldots,c\}}$ is \(c\), after a successive blow-ups of the varieties of the lowest index in each newly-obtained compatible system
$$
\widetilde{X}_1(D)^{(c)}\xrightarrow{\mu_{c}}\widetilde{X}_1(D)^{(c-1)}\xrightarrow{\mu_{c-1}} \cdots \xrightarrow{\mu_2} \widetilde{X}_1(D)^{(1)} \xrightarrow{\mu_1} X_1(D),
$$
their composition $\hat{\mu}: \widetilde{X}_1(D)^{(c)}\rightarrow X_1(D)$ is the \emph{canonical log resolution} of the ideal sheaf $\js_D$ associated to the compatible system $\{\tilde{D}_J\}_{\varnothing\subsetneq J\subset \{1,\ldots,c\}}$ in ${X}_1(D)$. The simultaneous log resolution for \(\js_I \) follows from \Cref{simultaneous} directly.   This concludes the proof of the lemma.
\end{proof}

As we saw in \eqref{Diagram}, we only need to find a simultaneous resolution for the indeterminacies of the rational maps \(X_1(D)\dashrightarrow X_1(D_i) \) for every \(i=1,\ldots,c \), which by \cref{equivalent simultaneous}  is equivalent to finding a simultaneous  log resolution for the ideal sheaves $\{\js_i\}_{i=1,\ldots,c}$ defined in \Cref{obstruction}. Observe that for each \(i=1,\ldots,c \),  the length of the subsystem \(\{\tilde{D}_J\}_{\varnothing\neq J\subset \{i\}^\complement}\) is equal to \(c-1\) by \Cref{resolution}. Thus by \eqref{sequence} the composition of $(c-1)$-steps of the blow-ups, $\mu_{c-1}\circ \mu_{c-2}\circ\cdots\circ \mu_1:\widetilde{X}_{1}(D)^{(c-1)}\rightarrow {X}_1(D)$, is a log resolution for each \(\js_i \). We will denote this simultaneous log resolution by $\mu:\widehat{X}_1(D)\rightarrow {X}_1(D)$, and call it \emph{the minimal  resolution}.\\

From this definition, it follows that  \Cref{resolution}  already implies  \cref{part1}. From \cref{lem:LocTrivial},  given  a family of smooth pairs \((\X,\ds)\stackrel{\rho}{\to} S\) as defined in \cref{def:smooth pair}, denoting by \((X_s,D_s)\) the fiber of \(\rho\), one obtains a locally analytically trivial family  \(\big(\pb\big(\Omega_{\xs/S}(\log \ds)\big),\bigcup_{I\subset \{1,\ldots,c\}}\tilde{\ds}^{\rm rel}_I\big)\stackrel{\rho_1}{\to} S 
	\) such that for any \(s\in S \),  the fiber of \(\rho_1\) is \( \big(X_{s,1}(D_s),\bigcup_{I\subset \{1,\ldots,c\}}\tilde{D}_{s,I}\big) \). \cref{part3}  follows from  \cref{lem:resolution family}.

 Let us now see how to obtain the announced fonctoriality property in \cref{part2}.
Let $Z$ be any regular subvariety of $X$, such that $\sum_{i=1}^{c}D_i|_Z$  is also a simple normal crossing divisor on $Z$. Denote by $E_i:=D_i|_Z$, and $E_I:=\bigcap_{i\in I}E_i$ for any $\varnothing\neq I\subset \{1,\ldots,c\}$. Note that ${Z}_1(E)$ is a regular subvariety of ${X}_1(D)$. Moreover,  for  $\tilde{E}_I\subset Z_1(E)$ associated to $E_I$ as defined in \Cref{meromorphic connection}, one has $\tilde{E}_I=Z_1(E)\cap \tilde{D}_I$. Thus  $\{ \tilde{D}_J\cap Z_1(E)\}_{\varnothing\neq J\subset \{1,\ldots,c\}}$ is a compatible system in $Z_1(E)$, with the index of $\tilde{D}_J\cap Z_1(E)$ equal to $\#J$. It then follows from \cref{functorial} that the restriction of the canonical log resolution  (resp. minimal  resolution)  associated to $\{ \tilde{D}_J\}_{\varnothing\neq J\subset \{1,\ldots,c\}}$ in $X_1(D)$, gives rise to   the canonical log resolution (resp. minimal  resolution)  associated to  $\{ \tilde{D}_J\cap Z_1(E)\}_{\varnothing\neq J\subset \{1,\ldots,c\}}$ in $Z_1(E)$. In particular, for any $i=1,\ldots,c$, let us denote by $\mu_Z:\widehat{Z}_1(E)\rightarrow Z_1(E)$ the minimal resolution of $Z_1(E)$, and one has the following commutative diagram
\begin{eqnarray}\label{commutative}
\xymatrix@!0{
	& X_1(D) \ar@{-->}[rrr]^{\gamma_i}
	& && X_1(D_i)
	\\
	&&&&\\
	\widehat{X}_1(D) \ar^{\mu}[uur] \ar^-{\nu_i}[uurrrr]
	& & &  
	\\
	& Z_1(E) \ar@{-->}[rrr]^{\gamma_{Z,i}} \ar@{^{(}->}[uuu] |!{[ul];[uuurr]}\hole
	& && Z_1(E_i)  \ar@{^{(}->}[uuu]
	\\
	&&&&\\
	\widehat{Z}_1(E)\ar@{^{(}->}[uuu] \ar^-(.8){\mu_Z}[uur]\ar^-{\nu_{Z,i}}[uurrrr]
	& &  &
}
\end{eqnarray}

Lastly, let us conclude by observing that  this resolution procedure resolves the indeterminacies of  the  different maps \(\gamma_I\) defined in \eqref{rational} step by steps in the following sense.
Say that 
$$D\supset  \sum_{i\in I_{1}}D_i\supset \sum_{i\in I_{2}}\supset \ldots\supset \sum_{i\in I_{c-1}}D_{i}$$ is a \emph{filtration of compatible divisors} if $I_{1}\supset I_{2}\supset\cdots\supset I_{c-1}$ is a sequence of subsets of $\{1,\ldots,c\}$ with the cardinalities $\#I_i=c-i$. From \Cref{def:log pair}, one then has a sequence of rational maps
\[
X_1(D)\dashrightarrow X_1\big(D(I_1)\big)\dashrightarrow \ldots\dashrightarrow X_1\big(D(I_{c-1})\big)\dashrightarrow X_1:=\pb(\Omega_X)
\]
Since $\#I_j^\complement=j$, by \Cref{simultaneous,resolution}, after taking $j$-blow-ups \[
\widetilde{X}_1(D)^{(j)}\xrightarrow{\mu_{j}}\widetilde{X}_1(D)^{(j-1)}\xrightarrow{\mu_{j-1}} \cdots \xrightarrow{\mu_2} \widetilde{X}_1(D)^{(1)}\xrightarrow{\mu_1} X_1(D),
\]
the composition  resolves $\js_{I_j^\complement}$ defined in \cref{obstruction}, and by  \cref{resolve},     it also resolves the indeterminacy of the rational map $ {X}_1(D)\dashrightarrow  {X}_1\big(D(I_j)\big)$. Namely, we have the following commutative diagram for any filtration of compatible divisors:
\[
\xymatrix{
	\widetilde{X}_{1}(D)^{(c)} \ar[d]^-{\mu_{c}}\ar[ddddrrrr]&&&&\\
	\!\!\!\!\!\!\!\!\!\!\!\!\!\!\!\!\!\!\!\!\! \widehat{X}_1(D):=\widetilde{X}_{1}(D)^{(c-1)} \ar[d]^-{\mu_{c-1}}\ar@/_2pc/[ddd]_{\mu}\ar[dddrrr]&&&&\\
	\vdots \ar[d]^{\mu_2} &&&\\
	\widetilde{X}_{1}(D)^{(1)} \ar[d]^-{\mu_1} \ar[dr] &\ddots&&\\
	X_1(D)\ar@{-->}[r]^-{\gamma_{I_1^\complement}}  \ar@{-->}@/_1.5pc/[rrr]_{\gamma_{I_{c-1}^\complement}} &  {X}_{1}\big(D(I_1)\big) \ar@{-->}[r]  & \cdots \ar@{-->}[r]  &  {X}_{1}\big(D(I_{c-1})\big) \ar@{-->}[r] & \pb(\Omega_X).
}
\]

\bibliography{biblio}

\begin{thebibliography}{10}

\bibitem{Ben11}
Olivier Benoist.
\newblock Le th\'eor\`eme de {B}ertini en famille.
\newblock {\em Bulletin de la Soci\'et\'e Math\'ematique de France},
  139(4):555--569, 2011.

\bibitem{BPW17}
Bo~Berndtsson, Mihai P\u{a}un, and Xu~Wang.
\newblock Algebraic fiber spaces and curvature of higher direct images.
\newblock {\em arXiv preprint arXiv:1704.02279}, 2017.

\bibitem{Bro11}
Damian Brotbek.
\newblock Differential equations as embedding obstructions and vanishing
  theorems.
\newblock {\em arXiv preprint arXiv:1111.5324}, 2011.

\bibitem{Bro16}
Damian Brotbek.
\newblock Symmetric differential forms on complete intersection varieties and
  applications.
\newblock {\em Math. Ann.}, 366(1-2):417--446, 2016.

\bibitem{Bro17}
Damian Brotbek.
\newblock On the hyperbolicity of general hypersurfaces.
\newblock {\em Publ. Math. Inst. Hautes \'Etudes Sci.}, 126:1--34, 2017.

\bibitem{BD15}
Damian Brotbek and Lionel Darondeau.
\newblock Complete intersection varieties with ample cotangent bundles.
\newblock {\em arXiv preprint arXiv:1511.04709}, 2015.

\bibitem{BR90}
P.~Br\"uckmann and H.-G. Rackwitz.
\newblock {$T$}-symmetrical tensor forms on complete intersections.
\newblock {\em Math. Ann.}, 288(4):627--635, 1990.

\bibitem{Bru16}
Yohan Brunebarbe.
\newblock Symmetric differentials and variations of hodge structures.
\newblock {\em Journal f{\"u}r die reine und angewandte Mathematik (Crelles
  Journal)}, 2016.

\bibitem{B-C17}
Yohan {Brunebarbe} and Benoit {Cadorel}.
\newblock {Hyperbolicity of varieties supporting a variation of Hodge
  structure}.
\newblock {\em ArXiv e-prints}, July 2017.

\bibitem{Cad16}
Benoit {Cadorel}.
\newblock {Symmetric differentials on complex hyperbolic manifolds with cusps}.
\newblock {\em ArXiv e-prints}, June 2016.

\bibitem{Dar15}
Lionel Darondeau.
\newblock On the logarithmic {G}reen-{G}riffiths conjecture.
\newblock {\em Int. Math. Res. Not. IMRN}, (6):1871--1923, 2016.

\bibitem{Deb05}
Olivier Debarre.
\newblock Varieties with ample cotangent bundle.
\newblock {\em Compositio Mathematica}, 141(6):1445--1459, 2005.

\bibitem{Dem97}
Jean-Pierre Demailly.
\newblock Algebraic criteria for {K}obayashi hyperbolic projective varieties
  and jet differentials.
\newblock In {\em Algebraic geometry---{S}anta {C}ruz 1995}, volume~62 of {\em
  Proc. Sympos. Pure Math.}, pages 285--360. Amer. Math. Soc., Providence, RI,
  1997.

\bibitem{Den16}
Ya~Deng.
\newblock Effectivity in the hyperbolicity-related problems.
\newblock {\em arXiv preprint arXiv:1606.03831}, 2016.

\bibitem{Den17}
Ya~Deng.
\newblock On the {D}iverio-{T}rapani conjecture.
\newblock {\em arXiv preprint arXiv:1703.07560}, 2017.

\bibitem{Div09}
Simone Diverio.
\newblock Existence of global invariant jet differentials on projective
  hypersurfaces of high degree.
\newblock {\em Math. Ann.}, 344(2):293--315, 2009.

\bibitem{DMR10}
Simone Diverio, Jo\"el Merker, and Erwan Rousseau.
\newblock Effective algebraic degeneracy.
\newblock {\em Invent. Math.}, 180(1):161--223, 2010.

\bibitem{ElG03}
Jawher El~Goul.
\newblock Logarithmic jets and hyperbolicity.
\newblock {\em Osaka J. Math.}, 40(2):469--491, 2003.

\bibitem{Gre77}
Mark~L. Green.
\newblock The hyperbolicity of the complement of {$2n+1$} hyperplanes in
  general position in {$P_{n}$} and related results.
\newblock {\em Proc. Amer. Math. Soc.}, 66(1):109--113, 1977.

\bibitem{Kob98}
Shoshichi Kobayashi.
\newblock {\em Hyperbolic complex spaces}, volume 318 of {\em Grundlehren der
  Mathematischen Wissenschaften [Fundamental Principles of Mathematical
  Sciences]}.
\newblock Springer-Verlag, Berlin, 1998.

\bibitem{Laz04II}
Robert Lazarsfeld.
\newblock {\em Positivity in algebraic geometry. {II}}, volume~49 of {\em
  Ergebnisse der Mathematik und ihrer Grenzgebiete. 3. Folge. A Series of
  Modern Surveys in Mathematics [Results in Mathematics and Related Areas. 3rd
  Series. A Series of Modern Surveys in Mathematics]}.
\newblock Springer-Verlag, Berlin, 2004.
\newblock Positivity for vector bundles, and multiplier ideals.

\bibitem{L-W12}
Steven S.~Y. Lu and J\"org Winkelmann.
\newblock Quasiprojective varieties admitting {Z}ariski dense entire
  holomorphic curves.
\newblock {\em Forum Math.}, 24(2):399--418, 2012.

\bibitem{McQ98}
Michael McQuillan.
\newblock Diophantine approximations and foliations.
\newblock {\em Inst. Hautes \'Etudes Sci. Publ. Math.}, (87):121--174, 1998.

\bibitem{Nak00}
Michael Nakamaye.
\newblock Stable base loci of linear series.
\newblock {\em Mathematische Annalen}, 318(4):837--847, 2000.

\bibitem{Nog81}
Junjiro Noguchi.
\newblock Lemma on logarithmic derivatives and holomorphic curves in algebraic
  varieties.
\newblock {\em Nagoya Math. J.}, 83:213--233, 1981.

\bibitem{Nog86}
Junjiro Noguchi.
\newblock Logarithmic jet spaces and extensions of de {F}ranchis' theorem.
\newblock In {\em Contributions to several complex variables}, Aspects Math.,
  E9, pages 227--249. Friedr. Vieweg, Braunschweig, 1986.

\bibitem{NW14}
Junjiro Noguchi and J\"org Winkelmann.
\newblock {\em Nevanlinna theory in several complex variables and {D}iophantine
  approximation}, volume 350 of {\em Grundlehren der Mathematischen
  Wissenschaften [Fundamental Principles of Mathematical Sciences]}.
\newblock Springer, Tokyo, 2014.

\bibitem{NWY07}
Junjiro Noguchi, J\"org Winkelmann, and Katsutoshi Yamanoi.
\newblock Degeneracy of holomorphic curves into algebraic varieties.
\newblock {\em J. Math. Pures Appl. (9)}, 88(3):293--306, 2007.

\bibitem{NWY08}
Junjiro Noguchi, J\"org Winkelmann, and Katsutoshi Yamanoi.
\newblock The second main theorem for holomorphic curves into semi-abelian
  varieties. {II}.
\newblock {\em Forum Math.}, 20(3):469--503, 2008.

\bibitem{NWY13}
Junjiro Noguchi, J\"org Winkelmann, and Katsutoshi Yamanoi.
\newblock Degeneracy of holomorphic curves into algebraic varieties {II}.
\newblock {\em Vietnam J. Math.}, 41(4):519--525, 2013.

\bibitem{Rou03}
Erwan Rousseau.
\newblock Hyperbolicit\'e du compl\'ementaire d'une courbe dans {$\Bbb P^2$}:
  le cas de deux composantes.
\newblock {\em C. R. Math. Acad. Sci. Paris}, 336(8):635--640, 2003.

\bibitem{Sak79}
Fumio Sakai.
\newblock Symmetric powers of the cotangent bundle and classification of
  algebraic varieties.
\newblock In {\em Algebraic geometry ({P}roc. {S}ummer {M}eeting, {U}niv.
  {C}openhagen, {C}openhagen, 1978)}, volume 732 of {\em Lecture Notes in
  Math.}, pages 545--563. Springer, Berlin, 1979.

\bibitem{Sch92}
Michael Schneider.
\newblock Symmetric differential forms as embedding obstructions and vanishing
  theorems.
\newblock {\em Journal of Algebraic Geometry}, 1(2):175--181, 1992.

\bibitem{Siu04}
Yum-Tong Siu.
\newblock Hyperbolicity in complex geometry.
\newblock In {\em The legacy of {N}iels {H}enrik {A}bel}, pages 543--566.
  Springer, Berlin, 2004.

\bibitem{Siu15}
Yum-Tong Siu.
\newblock Hyperbolicity of generic high-degree hypersurfaces in complex
  projective space.
\newblock {\em Invent. Math.}, 202(3):1069--1166, 2015.

\bibitem{V-Z02}
Eckart Viehweg and Kang Zuo.
\newblock Base spaces of non-isotrivial families of smooth minimal models.
\newblock In {\em Complex geometry ({G}\"ottingen, 2000)}, pages 279--328.
  Springer, Berlin, 2002.

\bibitem{Voi96}
Claire Voisin.
\newblock On a conjecture of {C}lemens on rational curves on hypersurfaces.
\newblock {\em J. Differential Geom.}, 44(1):200--213, 1996.

\bibitem{Voi98}
Claire Voisin.
\newblock A correction: ``{O}n a conjecture of {C}lemens on rational curves on
  hypersurfaces'' [{J}.\ {D}ifferential {G}eom.\ {\bf 44} (1996), no.\ 1,
  200--213; {MR}1420353 (97j:14047)].
\newblock {\em J. Differential Geom.}, 49(3):601--611, 1998.

\bibitem{Xie15}
Song-Yan {Xie}.
\newblock On the ampleness of the cotangent bundles of complete intersections.
\newblock {\itshape ArXiv e-prints 1510.06323}, october 2015.

\bibitem{Zuo00}
Kang Zuo.
\newblock On the negativity of kernels of {K}odaira-{S}pencer maps on {H}odge
  bundles and applications.
\newblock {\em Asian J. Math.}, 4(1):279--301, 2000.
\newblock Kodaira's issue.

\end{thebibliography}
\bibliographystyle{plain}

\end{document}